\documentclass{svjour3}                     
\smartqed  
\usepackage{graphicx}
\usepackage{amsmath, bm}
\usepackage{amsfonts}
\usepackage{amssymb}
\usepackage{amstext}
\usepackage{amscd}
\usepackage{enumerate}
\usepackage{pxfonts}
\usepackage{stmaryrd}
\makeatletter
\newcommand{\biggg}[1]{{\hbox{$\left#1\vbox to 20.5pt{}\right.\n@space$}}}
\newcommand{\Biggg}[1]{{\hbox{$\left#1\vbox to 23.5pt{}\right.\n@space$}}}
\newcommand{\bigggg}[1]{{\hbox{$\left#1\vbox to 26.5pt{}\right.\n@space$}}}
\newcommand{\Bigggg}[1]{{\hbox{$\left#1\vbox to 29.5pt{}\right.\n@space$}}}
\makeatother

\newcommand{\card}{\mbox{\rm{card}}}
\newcommand{\hausdorff}{\mbox{$\mathbin{\vdash\mspace{-8.0mu}\dashv}$}}
\newcommand{\grad}{\mbox{\rm{grad}}}
\newcommand{\ltree}[1]{\mbox{$\mathbf{#1}$}}
\newcommand{\isom}{\mbox{$\mathrel{\bar{\sim}}$}}
\newcommand{\EMn}[2]{\mbox{$#1^{({\rm EM}),#2}$}}
\newcommand{\ODp}[2]{\mbox{$#1^{({\rm ord}\,p),#2}$}}
\def\JelClassificationCodesName{{\bfseries JEL classification}\enspace}
\def\JelClassificationCodes#1{\par\addvspace
  \medskipamount{\rightskip=0pt plus1cm
\def\and{\ifhmode\unskip\nobreak\fi\ $\cdot$
}\noindent\JelClassificationCodesName\ignorespaces#1\par}}
\numberwithin{equation}{section}
\numberwithin{theorem}{section}
\newtheorem{myDef}[theorem]{Definition}
\newtheorem{myProp}[theorem]{Proposition}
\newtheorem{myLem}[theorem]{Lemma}
\newtheorem{myRem}[theorem]{Remark}
\newtheorem{myCor}[theorem]{Corollary}

\begin{document}
  
  \title{A new weak approximation
    scheme of stochastic differential equations and the Runge--Kutta
    method \thanks{This research was partly supported by the Ministry
      of Education, Science, Sports and Culture, Grant-in-Aid for Scientific
      Research (C), 15540110, 2003 and 18540113, 2006, 
      and by the 21st century COE program at
      Graduate School of Mathematics Sciences, the University of Tokyo. }}
  
  \titlerunning{New weak approximation scheme of SDEs}
  
  \author{Mariko Ninomiya  \and
    Syoiti Ninomiya
  }
  
  \authorrunning{M.~Ninomiya and S.~Ninomiya} 
  
  \institute{Mariko Ninomiya \at 
    Graduate School of Mathematical Sciences, The University of Tokyo, 
    3-8-1 Komaba, Meguro-ku, Tokyo 153-8914, Japan\\
    \email{mariko.nnmy@gmail.com}           
    \and
    Syoiti Ninomiya \at Center for Research in Advanced Financial
    Technology, Tokyo Institute of Technology,
    2-12-1 Ookayama, Meguro-ku,
    Tokyo 152-8552 Japan \\
    \email{ninomiya@craft.titech.ac.jp}
  }
  
  \date{Received:  / Accepted: }

  \maketitle
  
  \begin{abstract}
    The authors report on the construction of
    a new algorithm for the weak approximation of
    stochastic differential equations.  In this algorithm,
    an ODE-valued random variable
    whose average approximates
    the given stochastic differential equation is constructed by using
    the notion of free Lie algebra.  
    It is proved that the classical Runge--Kutta method for ODEs is directly 
    applicable to the drawn ODE from the random variable.
    In a numerical experiment, this is applied to 
    the problem of pricing Asian options under the Heston stochastic 
    volatility model.  
    Compared with some other methods, this algorithm
    gives significantly faster calculation times.
    \JelClassificationCodes{C63 \and G12}
  \end{abstract}
  
  \section{Introduction}\label{intro}
  \subsection{The problem and background}
  \subsubsection{The problem}
  Let $(\Omega, \mathcal{F}, P)$ be a probability space, $B^0(t)=t$,
  and $(B^1(t), \ldots, B^d(t))$ be a 
  $d$-dimensional standard Brownian motion.
  $C^{\infty}_b({{\mathbb R}}^N; {{\mathbb R}}^N)$
  denotes the set of $\mathbb{R}^N$-valued infinitely differentiable
  functions defined in $\mathbb{R}^N$ whose derivatives are all
  bounded. Our interest is in weak approximation, that is to say,
  approximation of $(P_tf)(x)=E[f(X(t,x))]$ 
  where $f\in C^\infty_b({\mathbb R}^N;{\mathbb R})$ and $X(t,x)$
  is a solution to the
  stochastic differential equation
  written in the Stratonovich form:
  \begin{equation}\label{SDE}
    X(t,x)=x+\sum_{i=0}^{d}\int_{0}^{t}V_i(X(s,x))\circ dB^i(s),
  \end{equation}
  where $V_i\in C^\infty_b({\mathbb R}^N;{\mathbb R}^N)$ for $i=0,1,\dots, d$.
  Here, $V_i\in C_b^\infty\left({\mathbb R}^N; {\mathbb R}^N\right)$
  is considered to be a vector field
  in the following way:
  \begin{equation*}
    V_if(x)=\sum_{j=1}^{N}V^{j}_{i}(x)\frac{\partial f}{\partial x_j}(x),
    \quad \text{for $f \in C_b^{\infty}({\mathbb R}^N; {\mathbb R})$}.
  \end{equation*} 
  It is well-known (e.~g.~\cite{ikeda-watanabe}) 
  that $E[f\left(X(1,x)\right)]$ 
  is equal to $u(1,x)$ where $u$ is 
  the solution to the following
  partial differential equation for $L=V_0+(1/2)\sum_{i=1}^dV_i^2$:
  \begin{equation}\label{PDE}
    \frac{\partial u}{\partial t}(t,x)=Lu, \quad u(0,x)=f(x).
  \end{equation}
  \subsubsection{Background}
  A number of studies on numerical calculations of
  this problem have been conducted as there is a
  great demand for it in various fields.
  One often encounters this type of calculation particularly in
  mathematical finance.
  For example, the price of a financial derivative
  written on the diffusion $X(t,x)$ is obtained by the calculation
  of $E[f(X(T,x))]$.
  \par
  There are two approaches to the problem:
  PDE approach and simulation.
  The former one involves solving the partial differential
  equation~\eqref{PDE} numerically.
  This method works only when $L$ is elliptic and the
  dimension is relatively small.
  We do not go into details on the subject here
  but refer to~\cite{LapeyrePardouxSentis}.
  These conditions
  are not necessarily satisfied in many practical
  problems so we are forced to take the other approach
  which is called the probabilistic method or simulation.
  In this paper, we focus on this approach.
  \par
  Usually, the Euler--Maruyama scheme is
  used to discretize $X(t,x)$ during simulations
  to weakly approximate $X(t,x)$.
  It is shown in~\cite{KusuokaNinomiya:2004}, \cite{ninomiya:2001a},
  \cite{ninomiya:2003}, and~\cite{shimizu:2002} that the new
  higher-order scheme introduced by Kusuoka in~\cite{kusuoka:2001aprx} 
  calculates some finance problems much faster
  than the Euler--Maruyama scheme.
  Lyons and Victoir extensively developed the scheme
  in~\cite{LyonsVictoir:2002} using the notion of free Lie algebra.
  Recent developments can be found
  in~\cite{bayer-teichmann} and~\cite{filipovic-tappe-teichmann}.
  \par
  We will discuss the reason why higher order schemes greatly
  improve the speed of numerical weak approximation in the later
  part of this paper (Section~\ref{sec:5}).
  
  \subsubsection{Our results}
  In this paper, we describe how successfully we constructed 
  in Theorem~\ref{contp_krk_main_thm1} and 
  Corollary~\ref{contp_krk_main_cor} a new higher order weak approximation
  scheme for a broad class of stochastic differential equations.
  This scheme owes a great deal to 
  the scheme shown in~\cite{kusuoka:2001aprx} and 
  to the cubature method on Wiener space introduced 
  in~\cite{LyonsVictoir:2002}.
  \par
  An intuitive explanation of the scheme is as follows.
  We construct the ODE (ordinary differential equation)-valued
  random variable whose average approximates the given
  stochastic differential equation.
  From this random variable, an ODE itself is able to be drawn at one time.  
  \par
  This scheme has a remarkable advantage that once an ODE is drawn,
  the conventional Runge--Kutta method can be applied so as to approximate 
  the ODE.  The approximating random variable is  
  constructed using Theorem~\ref{contp_krk_main_thm1}
  and Theorem~\ref{contp_krk_main_thm2}
  and can be approximated 
  by the Runge--Kutta method for ODEs via Theorem~\ref{krk_main_thm3}.
  \par
  We should note that another higher-order weak approximation
  method is introduced
  in~\cite{NinomiyaVictoir:2005}.
  Although the algorithm in~\cite{NinomiyaVictoir:2005} and the new
  method presented in this paper
  are based on the same scheme (\cite{kusuoka:2001aprx}
  \cite{LyonsVictoir:2002})
  and
  have many common features, algorithms themselves differ
  significantly.
  
  \subsection{Notation}
  Let $A\:=\{v_0, v_1,\dots, v_d\}$ 
  be an alphabet where $d\in{\mathbb Z}_{\geq 1}$ 
  and $A^*$ denote the set of all words consisting of the elements of
  $A$.  The empty word $1$ is the identity of $A^*$.  For 
  $u=v_{i_1}\cdots v_{i_n}\in A^*$, $|u|$ 
  and $\|u\|$ are defined by $|u|=n$ and
  $\|u\|=|u|+\card\left(\left\{k\,|\,i_k=0\right\}\right)$ 
  where $\card(S)$ denotes the cardinality of a set $S$.  
  Here, $\|\cdot\|$ is related to the scaling property
  of the Brownian motion.  
  $A^*_m$ and $A^*_{\leq m}$ denote $\{w\in A^* |\ |w|=m\}$ 
  and $\{w\in A^* |\ |w|\leq m\}$, respectively.  
  Let ${\mathbb R}\langle A\rangle$ be the
  ${\mathbb R}$-coefficient free algebra
  with basis $A^*$ and ${\mathbb R}\langle\langle A\rangle\rangle$ be 
  the set of all ${\mathbb R}$-coefficient formal series with basis $A^*$.  
  Then, ${\mathbb R}\langle A\rangle$ is a sub-${\mathbb R}$-algebra of 
  ${\mathbb R}\langle\langle A\rangle\rangle$.  We call an element
  of ${\mathbb R}\langle A\rangle$ a non-commutative polynomial.  
  $P\in {\mathbb R}\langle\langle A\rangle\rangle$ is written as
  \begin{equation*}
    P=\sum_{w\in A^*}\left(P, w\right)w\quad\text{or}
    \quad\sum_{w\in A^*}a_w w,
  \end{equation*}
  where $(P, w)=a_w\in {\mathbb R}$ denotes the coefficient of $w$.  
  Let 
  \begin{equation*}
    {\mathbb R}\langle A\rangle_m = \left\{P\in {\mathbb R}\langle A\rangle 
    \left|\ (P, w) = 0,\ \text{if $\|w\|\neq m$}  \right.\right\}.
  \end{equation*}  
  The algebra structure is defined as usual, \mbox{i.e.}
  \begin{equation*}
    \left(\sum_{w\in A^*}a_ww\right)\left(\sum_{w\in A^*}b_ww\right)
    =\sum_{\substack{w=uv\\w\in A^*}}a_ub_vw.\\
  \end{equation*}
  The Lie bracket is defined as $[x,y] = xy - yx$ for $x,y\in
  {\mathbb R}\langle\langle A\rangle\rangle$.  For $w=v_{i_1}\cdots v_{i_n}\in
  A^*$, $\mathfrak{r}(w)$ denotes
  $[v_{i_1},[v_{i_2},[\dots,[v_{i_{n-1}},v_{i_n}]\dots]]]$.  We define
  $\mathcal{L}_{\mathbb R}(A)$
  as the set of Lie polynomials in ${\mathbb R}\langle
  A\rangle$ and $\mathcal{L}_{\mathbb R}((A))$ as the set of Lie series.  
  This means that $\mathcal{L}_{\mathbb R}(A)$ is the smallest 
  sub-${\mathbb R}$-module
  of ${\mathbb R}\langle A\rangle$ including $A$ and is closed under the
  Lie bracket, and that $\mathcal{L}_{\mathbb R}((A))$ is
  the set of elements of
  ${\mathbb R}\langle\langle A\rangle\rangle$ whose homogeneous
  components belong to $\mathcal{L}_{\mathbb R}(A)$.  We note that 
  Lie polynomials correspond to vector fields while general polynomials 
  do not necessarily.  
  For $m\in\mathbb{Z}_{\geq 0}$,  let $j_m$ be a map defined by 
  \begin{equation*}
    j_m\left(\sum_{w\in A^*}a_w w\right) = \sum_{\|w\|\leq m}a_w w.
  \end{equation*}
  For arbitrary $P, Q\in {\mathbb R}\langle A\rangle$, the inner product
  $\langle P, Q\rangle$ is defined by
  \begin{equation*}
    \langle P, Q\rangle = \sum_{w\in A^*}(P, w)(Q, w).
  \end{equation*}
  Moreover we let $\|P\|_2 = \left(\langle P, P\rangle\right)^{1/2}$
  for $P\in \mathbb{R}\langle A\rangle$.  
  For $P\in {\mathbb R}\langle\langle A\rangle\rangle$ with $(P, 1)=0$, we
  can define $\exp(P)$ as $1+\sum_{k=1}^\infty P^k/k!$. 
  In addition, $\log(Q)$ 
  can be defined as $\sum_{k=1}^\infty(-1)^{k-1}(Q-1)^k/k$ for
  $Q\in{\mathbb R}\langle\langle A\rangle\rangle$ with $(Q,1) = 1$.
  Then the following relations hold:
  \begin{equation*}
    \log(\exp(P)) = P\quad\text{and}\quad
    \exp(\log(Q)) = Q.
  \end{equation*}
  \par
  By the natural identification
  $\mathbb{R}\langle\langle A\rangle\rangle
  \approx\mathbb{R}^\infty$, we can induce the direct product topology
  into $\mathbb{R}\langle\langle A\rangle\rangle$.  
  Then, $\mathbb{R}\langle\langle A\rangle\rangle$ becomes a Polish space by 
  the topology. We can also consider its Borel $\sigma$-algebra
  $\mathcal{B}(\mathbb{R}\langle\langle A\rangle\rangle)$,
  $\mathbb{R}\langle\langle A\rangle\rangle$-valued random variables,
  their expectations, and other notions as usual.
  \par
  Let $\Phi$ be the homomorphism between $\mathbb{R}\langle A\rangle$ 
  and the $\mathbb{R}$-algebra
  consisting of smooth differential operators over $\mathbb{R}^N$
  such that 
  \begin{equation}\label{Def:Phi}\begin{split}
      \Phi(1) &= \rm{Id}, \\
      \Phi(v_{i_1}\cdots v_{i_n}) &= V_{i_1}\cdots V_{i_n}
      \quad \text{for}\ i_1,\dots,i_n\in\{0,1,\dots,d\}.
  \end{split}\end{equation}
  
  Considering the scaling property of the Brownian motion, we 
  define the rescaling operator $\Psi_s$ depending on $\|\cdot\|$.  
  For $s\in\mathbb{R}_{>0}$, 
  $\Psi_s:\mathbb{R}\langle\langle A\rangle\rangle
  \longrightarrow\mathbb{R}\langle\langle A\rangle\rangle$
  is defined as follows:
  \begin{equation*}
    \Psi_s\left(
    \sum_{m=0}^\infty P_m
    \right)=\sum_{m=0}^\infty s^{m/2}P_m\quad \text{where
      $ P_m\in\mathbb{R}\langle
      A\rangle_m$.}
  \end{equation*}
  \par
  For a smooth vector field $V$, i.~e.~an element of
  $C_b^\infty\left({\mathbb R}^N;{\mathbb R}^N\right)$,
  $\exp\left(V\right)(x)$ denotes the solution at time $1$ of the ordinary
  differential equation
  \begin{equation*}
    \frac{dz_{t}}{dt}=V\left( z_{t}\right) ,\quad z_{0}=x.
  \end{equation*}
  We also define $\|V\|_{C^{n}}$ for $V\in C^{\infty}_b(\mathbb{R}^N;
  \mathbb{R}^N)$ as follows: 
  \begin{equation*}\begin{split}
      \|V\|&=\sup\left\{|V(x)|; \ x\in\mathbb{R}^N\right\} \\
      \left\Vert V^{(n)}\right\Vert&=
      \sup\left\{\left|V^{(n)}_{(x)}(U_1, U_2, \dots, U_n)\right|;
      \ x\in\mathbb{R}^N\, \text{and}\,
      \left| U_i\right| =1,\,\text{for}\, i=1,\dots, n\right\}\\
      \left\Vert V\right\Vert _{C^{n}}&=\sum_{i=0}^{n}\left\|V^{(i)}\right\|.
  \end{split}\end{equation*}
  Here $V^{(k)}$ denotes the $k$th order total differential
  of $V$, \mbox{i.e.} 
  \begin{equation*}
    V^{(n)}_{(x)}\left(U_1,U_2,\dots, U_n\right)=
    \sum_{i=1}^N\sum_{j_1=1}^N\cdots\sum_{j_n=1}^N
    \frac{\partial^nV_i}{\partial x_{j_1}\cdots\partial x_{j_n}}(x)
    U^{j_1}_1\cdots U^{j_n}_n e_i
  \end{equation*}
  where each $e_i$ denotes an $N$-dimensional unit vector,
  $\{e_1,\dots,e_N\}$ 
  forms an orthonormal basis of ${\mathbb R}^N$, and $U^j_k$ is 
  the $j$th component of $U_k\in{\mathbb R}^N$.
  
  \subsection{Main results}
  
  Since in this paper we deal with the operators that are not
  necessarily linear with respect to time $t$, we introduce
  the following definition:
  
  \begin{myDef}
    A map $g$ from $C_b^\infty({\mathbb R}^N; {\mathbb R}^N)$ to the set of all
    maps from ${\mathbb R}^N$ to ${\mathbb R}^N$ is called
    an integration scheme of order $m$
    if there exists a positive constant $C_m$ such that 
    \begin{equation}\label{RK_er}
      \sup_{x\in{\mathbb R}^N}\left\vert g(W)(x)
      -\exp{(W)}(x)\right\vert\leq C_m\|W\|_{C^{m+1}}^{m+1}
    \end{equation}
    for all $W\in C^\infty_b({\mathbb R}^N; {\mathbb R}^N)$.  
    Let $\mathcal{IS}(m)$ be 
    the set of all integration schemes of order $m$.
  \end{myDef}
  This definition is a generalization of the usual order
  of approximation.
  
  \begin{myDef}
    For $z_1, z_2\in\mathcal{L}_{\mathbb R}((A))$, 
    we define $z_2\hausdorff z_1$ as $\log(\exp(z_2)\exp(z_1))$.  Then 
    from the definition, for $z_1,z_2,z_3\in\mathcal{L}_{\mathbb R}((A))$,
    \begin{equation*}
      \left(z_1\hausdorff z_2\right)\hausdorff z_3
      =\log\left(\exp(z_1)\exp(z_2)\exp(z_3)\right)
      =z_1\hausdorff \left(z_2\hausdorff z_3\right),
    \end{equation*}
    and so we can write for $z_1,\dots, z_n\in\mathcal{L}_{\mathbb R}((A))$
    \begin{equation}
      z_1\hausdorff z_2\hausdorff\cdots\hausdorff z_n
      =\log\left(\exp(z_1)\cdots\exp(z_n)\right).
    \end{equation}
  \end{myDef}
  We notice that $z_2\hausdorff z_1\in\mathcal{L}_{\mathbb R}((A))$ 
  if $z_1, z_2\in\mathcal{L}_{\mathbb R}((A))$ from 
  the Baker--Campbell--Hausdorff formula(\cite{Bourbaki:Lie2}).
  
  The following are the main results.
  \begin{theorem}\label{contp_krk_main_thm1}
    Let $m\geq 1$, $M\geq 2$, and $Z_1,\dots, Z_M$ be 
    $\mathcal{L}_{\mathbb R}((A))$-valued random variables.  Assume that 
    $Z_1,\dots, Z_M$ satisfy the followings:
    \begin{gather}
      Z_i=j_mZ_i \quad\text{for}\ i=1,\dots,M,\\
      E\left[\left\|j_mZ_i\right\|_2\right]<\infty\quad \text{for}\ 
      i=1,\dots,M, \label{Z_condi1}\\
      E\left[\exp\left(a\sum_{j=1}^M
	\left\|\Phi\left(\Psi_s\left(Z_j\right)\right)\right\|
	_{C^{m+1}}\right)\right]<\infty
      \quad \text{for any $a>0$.}\label{Z_condi2} 
    \end{gather}
    Then for $p\in [1, \infty)$ and 
      arbitrary $g_1,\dots,g_M\in\mathcal{IS}(m)$, there exists a
      positive constant $C_{m,M}$ such that 
      \begin{multline}\label{contp_krk_main_thm1_eq}
	\left\|\sup_{x\in{\mathbb R}^N}
	\left|g_1\left(\Phi\left(\Psi_s\left(Z_1\right)\right)\right)\circ
	\cdots\circ g_M\left(\Phi\left(\Psi_s\left(Z_M\right)\right)\right)(x)
	\right.\right.\\
	-\exp\left(\Phi\left(\Psi_s\left(j_m
	\left(
	Z_M\hausdorff\cdots\hausdorff Z_1\right)
	\right)\right)\right)(x)\big|\bigg\|_{L^p}
	\leq C_{m,M}s^{(m+1)/2}
      \end{multline}
      for $s\in(0,1]$ where $C_{m,M}$ depends only on $m$ and $M$.  
    Here for functions $f$ and $g$, $f\circ g(x)$ denotes $f\left(g(x)\right)$ 
    as usual.  
  \end{theorem}
  
  For $i=1,\dots, d,$ and $j=1,\dots,M$, let $S^i_j$ be 
  ${\mathbb R}$-valued Gaussian random variables 
  and for $j, j^\prime=1,\dots,M$, let $c_j$ and $R_{jj^\prime}$ 
  be real numbers such that  
  \begin{equation}\label{Z_coeff_rvs}
    \sum_{j=1}^Mc_j  =1,\quad
    E\left[S_j^i\right]  =0,\quad\text{and}\quad
    E\left[S_j^i S_{j^\prime}^{i^\prime}\right]
    =R_{jj^\prime}\delta_{ii^\prime}
  \end{equation} 
  for $i, i^\prime=1,\dots, d$.  We let $S^0_j=c_j$ for convenience.  
  Taking~\eqref{PDE} into account, we let $Z_1,\dots, Z_M$ 
  be random variables such that $Z_j=c_jv_0+\sum_{i=1}^dS^i_jv_i$ 
  for $j=1,\dots, M$ and that 
  \begin{equation}\label{Z_j_condi}
    E\left[j_m\left(\exp\left(Z_1\right)\cdots\exp\left(
      Z_M\right)\right)\right]
    =j_m\left(\exp\left(v_0+\frac{1}{2}\sum_{i=1}^dv_i^2\right)\right).
  \end{equation}
  In usual ODE cases, this type of approximation technique is known as
  a splitting method (\cite{HairerLubichWanner:2006}).
  The stochastic versions of this technique are considered
  in~\cite{LyonsVictoir:2002} and~\cite{NinomiyaVictoir:2005}.

  \begin{myCor}\label{contp_krk_main_cor}
    Suppose that the following UFG condition is satisfied:
    \begin{description}
    \item[\textbf{(UFG)}] 
      There exist an integer $l$ and
      $\varphi_{u,u^\prime}\in C_b^\infty(\mathbb{R}^N;\mathbb{R})$
      which satisfy
      \begin{equation}\label{eq:UFG}
	\Phi(\mathfrak{r}(u))
	= \sum_{u^\prime\in A^*_{\leq l}\setminus\{1,v_0\}}
	\varphi_{u,u^\prime}\Phi(\mathfrak{r}(u^\prime))
      \end{equation}
      for any $u\in A^*\setminus\{1,v_0\}$.
    \end{description}
    For $j=1,\dots, M$ let $Z_j$ be $\mathcal{L}_{\mathbb R}((A))$-valued 
    random variables constructed as above 
    and define linear operators $Q_{(s)}$ for $s\in(0,1]$ by 
      \begin{equation}\label{Q_s_f}
	\left(Q_{(s)}f\right)(x)
	=E\left[f\left(g\left(\Phi\left(\Psi_s\left(Z_1\right)\right)\right)
	  \circ\cdots
	  \circ g\left(\Phi\left(\Psi_s\left(Z_M\right)\right)\right)(x)
	  \right)\right]
      \end{equation}
      where $f\in C_{b}^\infty({\mathbb R}^N;{\mathbb R})$
      and $g\in\mathcal{IS}(m)$.  Then 
      \begin{equation}\label{contp_krk_main_cor_order}
	\left\|
	P_sf-Q_{(s)}f
	\right\|_\infty\leq Cs^{(m+1)/2}
	\left\|
	\grad(f)
	\right\|_\infty
      \end{equation}
      where $C$ is a positive constant.
  \end{myCor}
  \begin{myRem}
    In~~\cite{kusuoka:2005:presentation}, it is shown that for the operator
    $Q_{(s)}$ defined above, there exists a constant $C$ and
    \begin{equation*}
      \left(P_sf\right)(x)-\left(Q_{(s)}f\right)(x)
      =Cs^{(m+1)/2}+O\left(s^{(m+3)/2}\right)
    \end{equation*}
    holds.
    This means that the Romberg extrapolation can be applied to
    our new algorithm.
  \end{myRem}
  
  The intuitive understanding is that once we find the random variables 
  $Z_1,\dots, Z_M$, we can numerically approximate $\exp(Z_i(\omega))$ by 
  applying the integration scheme $g_i$ repeatedly for each $i$ as seen 
  in~\eqref{contp_krk_main_thm1_eq} in Theorem~\ref{contp_krk_main_thm1}.  
  Therefore, our primary interest is in finding $Z_1,\dots, Z_M$. 
  
  \begin{theorem}\label{contp_krk_main_thm2}
    Let $m=5$ and $M=2$. Then \eqref{Z_j_condi} holds 
    if and only if 
    \begin{equation}\label{soln_m_5_n_2}
      \begin{split}
	c_1=\frac{\mp\sqrt{2\left(2u-1\right)}}{2},\quad
	c_2=1\pm\frac{\sqrt{2\left(2u-1\right)}}{2},\quad R_{11}=u&\\
	R_{22}=1+u\pm\sqrt{2\left(2u-1\right)},\quad
	R_{12}=-u\mp\frac{\sqrt{2\left(2u-1\right)}}{2}&
      \end{split}
    \end{equation}
    \noindent  
    for some $u\geq 1/2$. 
  \end{theorem}
  
  \begin{myRem}
    We can show that in the case where $m=7$ and $M=3$ there is no solution 
    to~\eqref{Z_j_condi}. 
  \end{myRem}
  
  Now that we have obtained the random variables satisfying~\eqref{Z_j_condi}, 
  we need a practical way of approximating these 
  integration schemes $g_1,\dots, g_M$.  We successfully extend applicability 
  of the general Runge--Kutta method to ODEs to find that it belongs to 
  $\mathcal{IS}(m)$.
  
  Let $A=\left(a_{ij}\right)_{i,j=1,\dots,K}$ with $a_{ij}\in{\mathbb R}$ 
  and $b={}^t\!\left(b_1,\dots,b_K\right)\in{\mathbb R}^K$.
  If $(A,b)$ satisfies 
  the $m$-th-order conditions defined as~\eqref{coeff_condis} 
  in Section~\ref{sec:3},  
  the $K$-stage Runge--Kutta method of order $m$ 
  in the sense of~\cite{Butcher:1987} can be
  written as follows: 
  \begin{equation}\label{original_RK}
    \begin{split}
      Y_i\left(W,s\right)
      &= y_0+s\sum_{j=1}^{K}a_{ij}W\left(Y_j(W,s)\right),\\
      Y(y_0;W,s) &= y_0+s\sum_{i=1}^{K}b_iW\left(Y_i(W,s)\right) 
    \end{split}
  \end{equation}
  for $W\in C^\infty_b\left({\mathbb R}^N;
  {\mathbb R}^N\right)$, $s\in{\mathbb R}_{>0}$, 
  and $y_0\in{\mathbb R}^N$.  
  Let $g(W)(y_0)$ be $Y(y_0;W,1)$.  We show that 
  $g$ belongs to $\mathcal{IS}(m)$ in Theorem~\ref{krk_main_thm3}.
  
  \begin{myRem}
    Our scheme is fundamentally different from 
    the class of numerical methods sometimes referred to as
    stochastic Runge--Kutta methods
    (\cite{BurrageBurrage:1998}\cite{roessler:2003}\cite{rumelin:1982}).
  \end{myRem}
  
  \section{Proof of Theorem~\ref{contp_krk_main_thm1}}\label{sec:1}
  We split the left-hand side of~\eqref{contp_krk_main_thm1_eq} as 
  \begin{equation}\label{krk_main_0}\begin{split}
      &\big\|\sup_{x\in{\mathbb R}^N}
      \left|g_1\left(\Phi_s\left(Z_1\right)\right)\circ
      \cdots\circ g_M\left(\Phi_s\left(Z_M\right)\right)(x)\right.\\
      &\quad\quad\quad\quad\quad -\left.\exp\left(\Phi_s\left(
      j_m\left(
      Z_M\hausdorff\cdots\hausdorff Z_1
      \right)
      \right)
      \right)(x)\right|\big\|_{L^p}\\
      &\leq\big\|\sup_{x\in{\mathbb R}^N}
      \left|\exp\left(\Phi_s\left(Z_1\right)\right)\circ
      \cdots\circ\exp\left(\Phi_s\left(Z_M\right)\right)(x)\right.\\
      &\quad\quad\quad\quad\quad -\left.\exp\left(\Phi_s\left(
      j_m\left(Z_M\hausdorff\cdots\hausdorff Z_1\right)
      \right)\right)(x)\right|
      \big\|_{L^p}\\
      &\quad\quad+\big\|\sup_{x\in{\mathbb R}^N}\left|g_1
      \left(\Phi_s\left(Z_1\right)\right)\circ
      \cdots\circ g_M\left(\Phi_s\left(Z_M\right)\right)(x)\right.\\
      &\quad\quad\quad\quad\quad -\left.\exp
      \left(\Phi_s\left(Z_1\right)\right)\circ
      \cdots\circ\exp\left(\Phi_s\left(Z_M\right)\right)(x)\right|\big\|_{L^p}.
  \end{split}\end{equation}
  Evaluation of each term of the right-hand side of~\eqref{krk_main_0} 
  will be given by Lemma~\ref{K_lem5} or~\eqref{krk_main_2} in this section. 
  \begin{myProp}\label{K_prop1}
    \ \\
    \begin{enumerate}[{\rm (1)}]
    \item For any $V\in C_{b}^\infty({\mathbb R}^N;{\mathbb R}^N)$, $f\in
      C^\infty({\mathbb R}^N; {\mathbb R})$, $x\in {\mathbb R}^N$ and $n\geq 1$,
      \begin{equation}\label{K_prop1_eq1}
	f\left(\exp(tV)(x)\right)=\sum_{k=0}^n\frac{t^k}{k!}
	\left(V^kf\right)(x)+\int_0^t\frac{(t-s)^n}{n!}\left(V^{n+1}f\right)
	\left(\exp(sV)(x)\right)ds.
      \end{equation}
      
    \item For all $z\in\mathcal{L}_{{\mathbb R}}((A))$ and $n,m\geq 1$,
      \begin{multline}\label{K_prop1_eq2}
	\sup_{x\in{\mathbb R}^N}
	\left|f\left(\exp\left(\Phi(j_mz)\right)(x)\right)-\sum_{k=0}^n
	\frac{1}{k!}\left(\Phi\left((j_mz)^k\right)f\right)(x)\right|\\
	\leq\frac{1}{(n+1)!}\left\|\Phi
	\left((j_mz)^{n+1}\right)f\right\|_\infty.
      \end{multline}
    \end{enumerate}
  \end{myProp}
  \begin{proof}
    Since we have 
    \begin{equation*}
      \frac{d}{dt}f\left(\exp(tV)(x)\right)=\left(Vf\right)
      \left(\exp\left(tV\right)(x)\right),
    \end{equation*}
    from the Taylor expansion, we obtain~\eqref{K_prop1_eq1} 
    by integration by parts and \eqref{K_prop1_eq2} can be derived 
    from~\eqref{K_prop1_eq1}.
    \qed
  \end{proof}
  
  \begin{myLem}\label{K_lem1}
    For all $n\geq 1$, there exists a constant $C_n>0$ such that 
    \begin{equation}\label{K_lem1_eq}
      \left\|\Phi(j_nz)f\right\|_\infty
      \leq C_n\left\|j_nz\right\|_2\left\|\grad(f)\right\|_{C^{n-1}}
    \end{equation}
    for all $z\in{\mathbb R}\langle\langle A\rangle\rangle$
    and $f\in C^\infty({\mathbb R}^N;{\mathbb R})$.
  \end{myLem}
  \begin{proof}
    Let $p_m$ be a map such that
    \begin{equation*}
      p_m\; :\;\sum_{|\alpha|=0}^\infty a_{\alpha}D^{\alpha}\longmapsto
      \sum_{|\alpha|=m}a_\alpha D^\alpha
    \end{equation*}
    where $\alpha$ is a multi-index, 
    $a_\alpha \in C^\infty_b({\mathbb R}^N; {\mathbb R})$, 
    and $D^\alpha =
    \frac{\partial^{\lvert \alpha\rvert}}{\partial
      x_1^{\alpha_1}\dots\partial x_N^{\alpha_N}}$.  
    Then we have 
    \begin{equation*}
      \Phi(w)=\sum_{i=1}^{|w|}p_i(\Phi(w)),
    \end{equation*}
    for $w\in A^*\setminus\{1\}$.  
    Since there exists a constant $C_{w,i}>0$ such that 
    \begin{equation*}
      \left\|p_i(\Phi(w))f\right\|_\infty
      \leq C_{w,i}\sup_{\substack{
	  \alpha\in \left({\mathbb Z}_{\geq 0}\right)^N \\
	  |\alpha|=i-1}}\left\|D^\alpha(\grad(f))\right\|_\infty,
    \end{equation*}
    we see that there exists a constant $C^\prime_n>0$ such that 
    \begin{equation*}\begin{split}
	\left\|\Phi(j_nz)f\right\|_\infty&
	\leq\sum_{\substack{w\in A^*\\ 1\leq\| w\|\leq n}}
	\left\|\Phi(w)f\right\|_\infty\left|\langle z,w\rangle\right|\\
	&\leq\sum_{\substack{w\in A^*\\ 1\leq\| w\|\leq n}}
	\sum_{i=1}^{|w|}C_{w,i}\left|\langle z, w\rangle\right|
	\sup_{\substack
	  {\alpha\in\left({\mathbb Z}_{\geq 0}\right)^N \\ |\alpha|=i-1}}
	\left\|D^{\alpha}(\grad(f))\right\|_\infty\\
	&\leq C_n^\prime\left\|j_nz\right\|_2
	\sup_{\substack{\alpha\in\left({\mathbb Z}_{\geq 0}\right)^N \\ 
	    |\alpha|\leq n-1}}\left\|D^\alpha (\grad(f))\right\|_\infty\\
	&\leq C_n\left\|j_nz\right\|_2\left\|\grad(f)\right\|_{C^{n-1}}
    \end{split}\end{equation*}
    where
    $C_n^\prime=\card\left(
    \left\{w\in A^*\,\vert\,1\leq\|w\|\leq n\right\}\right)
    \sup_{\substack{w\in A^*\\1\leq\|w\|\leq n}}
    (\sum_{i=1}^{|w|}C_{w,i})$.
    \qed
  \end{proof}
  
  For simplicity of notation, we let 
  $\Phi_s(y)$ denote $\Phi(\Psi_s(y))$ for an element 
  $y\in\mathcal{L}_{\mathbb R}((A))$ in the following part.
  \begin{myLem}\label{K_lem2}
    \begin{enumerate}[{\rm (1)}]
    \item There exists a constant $C_{m,1}>0$ 
      such that 
      \begin{multline}\label{K_lem2_eq1}
	\sup_{x\in{\mathbb R}^N}
	\left|f\left(\exp\left(\Phi_s\left(j_mz\right)\right)(x)\right)
	-\left(\Phi_s\left(j_m\exp\left(j_mz\right)\right)f\right)(x)
	\right|\\
	\leq C_{m,1}s^{(m+1)/2}
	\left(1+\left\|j_mz\right\|_2\right)^{m+1}
	\left\|\grad(f)\right\|_{C^{m(m+1)-1}}
      \end{multline}
      for $z\in\mathcal{L}_{\mathbb R}((A))$. 
    \item There exists a constant $C_{m,M}>0$ where $M\in{\mathbb Z}_{\geq 2}$ 
      such that 
      \begin{multline}\label{K_lem2_eq2}
	\sup_{x\in{\mathbb R}^N}
	\left|f\left(\exp\left(\Phi_s\left(
	j_m\left(\left(j_mz_M\right)\hausdorff
	\cdots\hausdorff\left(
	j_mz_1\right)\right)\right)\right)(x)\right)\right.\\
	-\left. \left(\Phi_s\left(j_m\exp\left(j_m\left(\left(j_mz_M\right)
	\hausdorff
	\cdots\hausdorff\left(j_mz_1\right)\right)\right)\right)f\right)(x)
	\right|\\
	\leq C_{m,M}s^{(m+1)/2}
	\left(1+\sum_{i=1}^M\left\|j_mz_i\right\|_2\right)^{m+1}
	\left\|\grad(f)\right\|_{C^{m(m+1)-1}}
      \end{multline}
      for $z_1, \dots, z_M\in\mathcal{L}_{\mathbb R}((A))$. 
    \end{enumerate}
  \end{myLem}
  
  \begin{proof}
    From the fact that for $z\in\mathcal{L}_{\mathbb R}((A))$
    \begin{equation*}
      j_m\left(\exp\left(j_mz\right)\right)=\sum_{k=0}^{m}
      \frac{1}{k !}\left(j_mz\right)^{k}-\sum_{k=2}^{m}
      \frac{1}{k !}(j_{m(m+1)}-j_m)
      \left(\left(j_mz\right)^{k}\right) 
    \end{equation*}
    and from \eqref{K_prop1_eq2} in Proposition~\ref{K_prop1}, we see that 
    \begin{multline}\label{K_lem2_2}
      \left|f\left(\exp\left(\Phi(j_mz)\right)(x)\right)
      -\left(\Phi\left(j_m\left(\exp(j_mz)\right)\right)f\right)(x)\right|\\
      \leq \frac{1}{(m+1)!}\left\|\Phi\left(\left(j_mz\right)^{m+1}\right)f
      \right\|_\infty
      +\left|\sum_{k=2}^{m}\frac{1}{k !}
      \left(\Phi\left((j_{m(m+1)}-j_m)
      \left((j_mz)^{k}\right)\right)
      f\right)(x)\right|.
    \end{multline}
    Since we have 
    \begin{equation*}
      \left(j_mz\right)^{m+1}=\left(j_{m(m+1)}-j_m\right)
      \left(j_mz\right)^{m+1}, 
    \end{equation*}  
    the followings can be derived by applying Lemma~\ref{K_lem1}:
    \begin{equation}\label{K_lem2_3}\begin{split}
	&\left|f\left(\exp\left(\Phi\left(j_mz\right)\right)(x)\right)
	-\left(\Phi\left(j_m\left(\exp\left(j_mz\right)\right)\right)f
	\right)(x)\right|\\
	&\qquad\qquad\qquad\leq \sum_{k=2}^{m+1}\frac{1}{k !}
	\left\|\Phi\left(\left(j_{m(m+1)}-j_m\right)\left(
	\left(j_mz\right)^{k}\right)\right)f\right\|_\infty\\
	&\qquad\qquad\qquad\leq C_m\sum_{k=2}^{m+1}
	\left\|\left(j_{m(m+1)}-j_m\right)
	\left(j_mz\right)^k\right\|_2
	\left\|\grad(f)\right\|_{C^{m(m+1)-1}}\\
	&\qquad\qquad\qquad\leq C_{m,1}\left(
	1+\left\|j_mz\right\|_2\right)^{m+1}
	\|\grad(f)\|_{C^{m(m+1)-1}}
    \end{split}\end{equation}
    where $C_m$ and $C_{m,1}$ are positive constants.  
    Thus~\eqref{K_lem2_eq1} is proved.\\
    Taking $(j_mz_M)\hausdorff\cdots\hausdorff (j_mz_1)$ 
    as $z$ above and evaluating by  
    \begin{equation*}
      \sum_{k=2}^{m+1}\left\|\left(j_{m(m+1)}-j_m\right)
      \left(j_m\left(\left(j_mz_M\right)\hausdorff\cdots\hausdorff 
      \left(j_mz_1\right)\right)\right)^k\right\|_2
      \leq C_{m,M}
      \left(1+\sum_{i=1}^M\left\|j_mz_i\right\|_2\right)^{m+1},
    \end{equation*}
    we obtain~\eqref{K_lem2_eq2}.
    \qed
  \end{proof}
  \begin{myLem}\label{K_lem3}
    There exists a constant $C_{m,M}>0$ such that 
    \begin{multline}\label{K_lem3_eq}
      \sup_{x\in{\mathbb R}^N}
      \left|f\left(\exp\left(\Phi_s\left(j_mz_1\right)\right)\circ\cdots
      \circ\exp\left(\Phi_s\left(j_mz_M\right)\right)(x)\right)\right.\\
      -\left(
      \Phi_s\left(j_m\exp\left(j_m\left(
      \left(j_mz_M\right)\hausdorff\cdots\hausdorff\left(j_mz_1\right)\right)
      \right)\right)f\right)(x)\big|\\
      \leq C_{m,M}s^{(m+1)/2}\sum_{i=1}^{M}
      \left(1+\left\|j_mz_i\right\|_2\right)^{m+1}
      \left\|\grad(f)\right\|_{C^{m(m+M)-1}}
    \end{multline}
    for $z_1, \dots, z_M\in\mathcal{L}_{\mathbb R}((A))$. Here 
    $C_{m,M}$ depends on $m$ and $M$.
  \end{myLem}
  \begin{proof}
    We prove the lemma by induction on $M$.
    When $M=1$, \eqref{K_lem2_eq1} and~\eqref{K_lem3_eq} are equivalent. 
    Assume that~\eqref{K_lem3_eq} holds for $M$.
    Splitting the left-hand side of~\eqref{K_lem3_eq} for $M+1$ as  
    \begin{equation*}\begin{split}
	&\left|f\left(\exp\left(\Phi_s\left(j_mz_1\right)\right)\circ\cdots
	\circ\exp\left(\Phi_s\left(j_mz_{M+1}\right)\right)(x)\right)\right.\\
	&\quad\quad\quad\quad\quad -\left.\left(
	\Phi_s\left(j_m\exp\left(j_m\left(
	\left(j_mz_{M+1}\right)\hausdorff\cdots
	\hausdorff\left(j_mz_1\right)\right)\right)\right)f
	\right)(x)\right|\\
	&\leq\left|f\left(\exp\left(\Phi_s\left(j_mz_1\right)\right)\circ
	\cdots\circ\exp\left(\Phi_s\left(j_mz_{M+1}\right)\right)(x)\right)
	\right. \\
	&\qquad\qquad\quad\left.
	-\left(\Phi_s\left(j_m\exp\left(j_m\left(\left(j_mz_M\right)\hausdorff
	\cdots\hausdorff\left(j_mz_1\right)\right)\right)\right)f\right)
	\left(\exp\left(\Phi_s\left(j_mz_{M+1}\right)\right)(x)\right)\right|\\
	&\quad +\left|
	\left(
	\Phi_s\left(j_m\exp\left(j_m\left(\left(j_mz_M\right)\hausdorff
	\cdots\hausdorff\left(j_mz_1\right)\right)\right)\right)f
	\right)
	\left(\exp\left(\Phi_s\left(j_mz_{M+1}\right)\right)(x)\right)
	\right. \\
	&\qquad\qquad\quad\left.
	-\left(\Phi_s\left(j_m\exp\left(j_m\left(\left(j_mz_{M+1}\right)
	\hausdorff
	\cdots\hausdorff\left(
	j_mz_{1}\right)\right)\right)\right)f\right)(x)\right|,
    \end{split}\end{equation*}
    we can apply the induction hypothesis and~\eqref{K_lem2_eq1} with 
    $\Phi_s\left(j_m\exp\left(\left(j_mz_M\right)\hausdorff
    \cdots\hausdorff\left(j_mz_1\right)\right)\right)f$  
    instead of $f$ obtaining
    \begin{multline*}
      \sup_{x\in{\mathbb R}^N}
      \left|f\left(\exp\left(\Phi_s\left(j_mz_1\right)\right)\circ\cdots
      \circ\exp\left(\Phi_s\left(j_mz_{M+1}\right)\right)(x)\right)\right.\\
      -\left(
      \Phi_s\left(j_m\exp\left(j_m\left(\left(j_mz_{M+1}\right)\hausdorff\cdots
      \hausdorff\left(j_mz_1\right)\right)\right)\right)f
      \right)(x)\big|\\
      \leq C_1s^{(m+1)/2}\left(\sum_{i=1}^M
      \left(1+\left\|j_mz_i\right\|_2\right)^{m+1}
      +\left(1+\left\|j_mz_{M+1}\right\|_2\right)^{m+1}\right)
      \left\|\grad(f)\right\|_{C^{m(m+M+1)-1}},
    \end{multline*}
    where $C_1>0$ is a constant depending on $m$ and $M$.  
    Hence, \eqref{K_lem3_eq} is proved.
    \qed
  \end{proof}
  From Lemma~\ref{K_lem2} and~\ref{K_lem3}, we have the following result. 
  \begin{myLem}\label{K_lem4}
    For all $m\geq 1$, there exists a constant $C_{m,M}>0$ such that 
    \begin{multline}\label{K_lem4_eq}
      \sup_{x\in{\mathbb R}^N}\left|f\left(
      \exp\left(\Phi_s\left(j_mz_1\right)\right)\circ\cdots
      \circ\exp\left(\Phi_s\left(j_mz_M\right)\right)(x)
      \right)\right.\\
      -f\left(\exp\left(\Phi_s\left(j_m\left((j_mz_M)\hausdorff\cdots
      \hausdorff(j_mz_1)\right)\right)\right)(x)\right)\big|\\
      \leq C_{m,M}s^{(m+1)/2}
      \sum_{i=1}^M\left(1+\left\|j_mz_i\right\|_2\right)^{m+1}
      \left\|\grad(f)\right\|_{C^{m(m+M)-1}}
    \end{multline}
    for all $s\in (0, 1]$, $z_1,\dots, z_M\in\mathcal{L}_{\mathbb R}((A))$, 
      and $f\in C^\infty ({\mathbb R}^N; {\mathbb R})$. 
  \end{myLem}
  
  \begin{myLem}\label{K_lem5}
    Let $Z_1,\dots, Z_M$ be $\mathcal{L}_{\mathbb R}((A))$-valued
    random variables such that for $m\geq 1$,
    $E\left[\left\|j_mZ_i\right\|_2\right]<\infty$ for $i=1,\dots,M$.
    Then, for $p\in[1, \infty)$ there exists a constant 
      $C_{m,M}>0$ such that 
      \begin{multline}\label{K_lem5_eq}
	\left\|\sup_{x\in{\mathbb R}^N}\left|
	\exp\left(\Phi_s\left(j_mZ_1\right)\right)\circ\cdots\circ
	\exp\left(\Phi_s\left(j_mZ_M\right)\right)(x)\right.\right.\\
	-\exp\left(\Phi_s\left(j_m\left((j_mZ_M)\hausdorff
	\cdots\hausdorff (j_mZ_1)\right)\right)\right)(x)\big|\bigg\|_{L^p}
	\leq C_{m,M}s^{(m+1)/2}
      \end{multline}
      for any $s\in (0, 1]$.
  \end{myLem}
  
  \begin{proof}
    If for $i\in\{1,\dots, N\}$, $f((x^1,\dots, x^N))=x^i$, 
    then $\left\|\grad(f)\right\|_{C^{m(m+M)-1}}=1$ for all $m\geq 1$. 
    Therefore, applying Lemma~\ref{K_lem4} for this $f$, 
    we obtain~\eqref{K_lem5_eq}.
    \qed
  \end{proof}
  We note that in~\cite{strichartz:1987} 
  a similar result to this Lemma is obtained.
  
  We now start discussion about the latter term of the right-hand side 
  of ~\eqref{krk_main_0}. 
  \begin{myProp}\label{prop_Gronwall}
    There exists a constant $C>0$ such that 
    \begin{equation}\label{prop_Gronwall_eq}
      \left|g(W)(x)-g(W)(y)\right|\leq
      C\left\|W\right\|^{m+1}_{C^{m+1}}+|x-y|\exp\left(\|W\|_{C^1}\right)
    \end{equation}
    for $g\in\mathcal{IS}(m)$ and
    $W\in C^\infty_b({\mathbb R}^N;{\mathbb R}^N)$. 
  \end{myProp}
  \begin{proof}
    Since Gronwall's inequality gives  
    \begin{equation*}
      \left|\exp\left(W\right)(x)-\exp\left(W\right)(y)\right|\leq
      \left|x-y\right|\exp\left(\left\|W\right\|_{C^1}\right),
    \end{equation*}
    \eqref{prop_Gronwall_eq} can be derived.
    \qed
  \end{proof}
  
  Since $g_i\in\mathcal{IS}(m)$ and each $Z_i$ satisfies ~\eqref{Z_condi2}, 
  we see 
  that for some $C_1>0$, 
  \begin{multline}\label{krk_main_1}
    \left\|\sup_{x\in{\mathbb R}^N}\left|
    g_M\left(\Phi_s\left(Z_M\right)\right)(x)
    -\exp\left(\Phi_s\left(Z_M\right)\right)(x)\right|\right\|_{L^p}\\
    \leq \left\|C_m\left\|
    \Phi_s\left(Z_M\right)
    \right\|_{C^{m+1}}^{m+1}\right\|_{L^p}
    \leq C_1s^{(m+1)/2}.
  \end{multline}
  From this fact and Proposition~\ref{prop_Gronwall}, 
  there exists a constant $C_4>0$ such that
  \begin{equation*}
    \begin{split}
      &\left\|\sup_{x\in{\mathbb R}^N}
      \left|g_{M-1}\left(\Phi_s\left(Z_{M-1}\right)\right)\circ
      g_M\left(\Phi_s\left(Z_M\right)\right)(x)
      -\exp\left(\Phi_s\left(Z_{M-1}\right)\right)\circ
      \exp\left(\Phi_s\left(Z_M\right)\right)(x)\right|\right\|_{L^p}\\
      &\leq \left\|\sup_{x\in{\mathbb R}^N}
      \left|g_{M-1}\left(\Phi_s\left(Z_{M-1}\right)\right)\circ
      \exp\left(\Phi_s\left(Z_M\right)\right)(x)
      -\exp\left(\Phi_s\left(Z_{M-1}\right)\right)\circ
      \exp\left(\Phi_s\left(Z_M\right)\right)(x)\right|\right\|_{L^p}\\
      &\quad +\left\|\sup_{x\in{\mathbb R}^N}
      \left|g_{M-1}\left(\Phi_s\left(Z_{M-1}\right)\right)\circ
      g_M\left(\Phi_s\left(Z_M\right)\right)(x)
      -g_{M-1}\left(\Phi_s\left(Z_{M-1}\right)\right)\circ
      \exp\left(\Phi_s\left(Z_M\right)\right)(x)\right|\right\|_{L^p}\\
      &\leq C_2s^{(m+1)/2}\\
      &\quad+
      \left\|
      C_3\left\|\Phi_s\left(Z_{M-1}\right)\right\|_{C^{m+1}}^{m+1}
      +\sup_{x\in{\mathbb R}^N}\left|g_M\left(\Phi_s\left(Z_M\right)\right)(x)
      -\exp\left(\Phi_s\left(Z_M\right)\right)(x)\right|
      \exp\left(\left\|
      \Phi_s\left(Z_{M-1}\right)
      \right\|_{C^1}\right)\right\|_{L^p}\\
      &\leq C_4s^{(m+1)/2}
  \end{split}\end{equation*}
  where $C_2$ and $C_3$ are positive constants.  Inductively,
  \begin{multline}\label{krk_main_2}
    \bigg\|\sup_{x\in{\mathbb R}^N}
    \big|g_1\left(\Phi_s\left(Z_1\right)\right)\circ
    \cdots\circ g_M\left(\Phi_s\left(Z_M\right)\right)(x)\\
    -\exp\left(\Phi_s\left(Z_1\right)\right)\circ
    \cdots\circ\exp\left(\Phi_s\left(Z_M\right)\right)(x)
    \big|\bigg\|_{L^p}
    \leq C_5s^{(m+1)/2}
  \end{multline}
  where $C_5>0$.
  
  Lemma~\ref{K_lem5} and~\eqref{krk_main_2} complete the proof.

  \section{Construction of the $\mathcal{L}_\mathbb{R}((A))$-valued 
    random variables $Z_1,\dots, Z_M$}\label{sec:2}
  
  \begin{myLem}\label{lem_coeff_repr}
    For $i=1,\dots,M$, let $Y_i$  
    be Gaussian random variables such that 
    \begin{equation*}
      E\left[Y_i\right]=0\quad\text{and}\quad E\left[Y_iY_j\right]=R(i,j),
      \quad \text{for}\ i,j=1,\dots, M
    \end{equation*}
    where $R(i,j)\in{\mathbb R}$.
    Moreover, for $i=1,\dots,M$ let $m_i\in{\mathbb Z}_{\geq 0}$ 
    be such that $\sum_{i=1}^Mm_i$ is even.  Then 
    we have 
    \begin{equation}\label{lem_coeff_repr_eq}
      E\left[Y_1^{m_1}\cdots Y_M^{m_M}\right]
      =\sum_{\left\{d_{ij}\right\}_{1\leq i\leq j\leq M}
	\in e\left(m_1,\dots, m_M\right)}
      2^{-\sum_{i=1}^Md_{ii}}
      \frac{\prod_{i=1}^M\left(m_i!\right)}
	   {\prod_{1\leq i\leq j\leq M}\left(d_{ij}!\right)}
	   \prod_{1\leq i\leq j\leq M}R(i,j)^{d_{ij}}
    \end{equation}
    where $e\left(m_1,\dots, m_M\right)$ is a set of 
    $\{d_{ij}\}_{1\leq i\leq j\leq M}$ satisfying that 
    $d_{ij}\in{\mathbb Z}_{\geq 0}$ and that 
    \begin{equation*}   
      \sum_{1\leq j<i}d_{ji}+2d_{ii}+
      \sum_{i<j \leq M}d_{ij}=m_i 
    \end{equation*}
    for $i=1,\dots,M$.  
  \end{myLem}
  \begin{proof}
    Let $l=\sum_{i=1}^Mm_i$.  We have 
    \begin{equation}\label{lem_coeff_repr_1}
      \begin{split}
	E\left[Y^{m_1}_1\cdots Y^{m_M}_M\right]
	&=\left. E\left[\frac{\partial^l}
	  {\partial z^{m_1}_1\cdots\partial z^{m_M}_M}
	  \exp\left(\sum_{i=1}^Mz_iY_i\right)\right]\right|_{z=0}\\
	&=\left. \frac{\partial^l}{\partial z^{m_1}_1\cdots\partial z^{m_M}_M}
	\left(
	\exp\left(\frac{1}{2}\sum_{1\leq i,j\leq M}R(i,j)z_iz_j\right)
	\right)
	\right|_{z=0}\\
	&=\frac{1}{2^{l/2}(l/2)!}
	\left. \frac{\partial^l}{\partial z^{m_1}_1\cdots\partial z^{m_M}_M}
	\left(\sum_{1\leq i,j\leq M}R(i,j)z_iz_j\right)^{l/2}\right|_{z=0} 
    \end{split}\end{equation}
    where $z=(z_1,\dots,z_M)\in{\mathbb R}^M$.  
    
    Let 
    \begin{equation}\label{lem_coeff_repr_2}
      e_l=\left\{\left\{d_{ij}\right\}_{1\leq\ i\leq j\leq M}\,\middle\vert\,
      d_{ij}\in{\mathbb Z}_{\geq 0}\, \text{and}\, 
      \sum_{1\leq i\leq j\leq M}d_{ij}=\frac{l}{2}\right\}. 
    \end{equation}
    Then 
    \begin{equation}\label{lem_coeff_repr_3}
      \begin{split}
	&\left(\sum_{1\leq i,j\leq M}R(i,j)z_iz_j\right)^{l/2}
	=\left(
	\sum_{i=1}^MR(i,i)z_i^2+2\sum_{1\leq i<j\leq M}R(i,j)z_iz_j
	\right)
	^{l/2}\\
	&=\sum_{\{d_{ij}\}_{1\leq i\leq j\leq M}\in e_l}
	\frac{\left(l/2\right)!}
	     {\prod_{1\leq i\leq j\leq M}\left(d_{ij}!\right)}
	     \prod_{i=1}^M\left(R(i,i)z_i^2\right)^{d_{ii}}
	     \prod_{1\leq i<j\leq M}\left(2R(i,j)z_iz_j\right)^{d_{ij}}\\
	     &=\sum_{\{d_{ij}\}_{1\leq i\leq j\leq M}\in e_l}
	     \frac{\left(l/2\right)!}
		  {{\prod_{1\leq i\leq j\leq M}\left(d_{ij}!\right)}}
		  2^{\sum_{1\leq i<j\leq M}d_{ij}}
		  \left(\prod_{1\leq i\leq j\leq M}R(i,j)^{d_{ij}}\right)
		  \left(
		  \prod_{i=1}^Mz_i^{(\sum_{1\leq j<i}d_{ji}+2d_{ii}
		    +\sum_{i<j\leq M}d_{ij})}
		  \right).
    \end{split}\end{equation}
    Hence 
    \begin{multline}\label{lem_coeff_repr_4}
      \left. \frac{\partial^l}{\partial z^{m_1}_1\cdots\partial z^{m_M}_M}
      \left(\sum_{1\leq i,j\leq M}R(i,j)z_iz_j\right)^{l/2}\right|_{z=0}\\
      =\left(m_1!\cdots m_M!\right)\left(\frac{l}{2}\right)!
      \sum_{\left\{d_{ij}\right\}_{1\leq i\leq j\leq M}
	\in e(m_1,\dots, m_M)}
      \frac{2^{\sum_{1\leq i\leq j\leq M}d_{ij}}}
	   {\prod_{1\leq i\leq j\leq M}\left(d_{ij}!\right)}
	   \prod_{1\leq i\leq j\leq M}R(i,j)^{d_{ij}}. 
    \end{multline}
    Since we have from the definition of $e(m_1,\dots, m_M)$ that 
    \begin{equation*}
      \sum_{1\leq i<j\leq M}d_{ij}=\frac{l}{2}-\sum_{i=1}^Md_{ii} 
    \end{equation*}
    for $\{d_{ij}\}\in e(m_1,\dots, m_M)$, \eqref{lem_coeff_repr_eq} 
    is derived from~\eqref{lem_coeff_repr_1} 
    and~\eqref{lem_coeff_repr_4}.
    \qed
  \end{proof}

  We need a simple representation 
  of the coefficient of each 
  $v_{i_1}v_{i_2}\cdots v_{i_\ell}$ in
  $E\left[\exp(Z_1)\cdots\exp(Z_M)\right]$ where
  $(i_1,\dots, i_\ell)\in\{0,1,\dots, d\}^\ell$ 
  and
  $Z_1,\dots Z_M$ are
  $\mathcal{L}_\mathbb{R}((A))$-valued random variables
  constructed with Gaussian random variables satisfying~\eqref{Z_coeff_rvs}.

  For $\ell, M\in\mathbb{Z}_{>0}$, let 
  $\mathcal{K}_\ell(M)=\left\{\vec{k}=(k_1,\dots, k_M)\in
  \left({\mathbb Z}_{\geq 0}\right)^M\,
  \middle\vert\,k_1+\cdots +k_M=\ell\right\}$.  
  For $w=v_{i_1}\cdots v_{i_\ell}\in A^*$, let 
  $N^w\, :\, \{0,1,\dots, d\}\times\{1,\dots, M\}\times \mathcal{K}_\ell(M)
  \longrightarrow {\mathbb Z}_{\geq 0}$ 
  be a function such that 
  \begin{equation*}
    N^w\left(i,j,\vec{k}\right)
    =\card\left(\left\{r\, \middle\vert\, i_r=i \quad\text{for}\quad 
    k_1+\cdots+k_{j-1}+1\leq r\leq k_1+\cdots+k_j\right\}\right).
  \end{equation*}

  \begin{theorem}\label{thm_coeff_repr}
    Let $w=v_{i_1}v_{i_2}\cdots v_{i_\ell}\in A^*$ and 
    $n^w(i)=\card\left(\{j\in\{1,\dots,\ell\}\, \middle\vert\, 
    i_j=i\}\right)$ for 
    $i=1,\dots, d$.  
    Then the coefficient of $w$, $C(w)$, in $E[\exp(Z_1)\cdots\exp(Z_M)]$ 
    becomes as follows: \\
    
    \noindent If $n^w(i)$ is odd for some $i\in\left\{1,\dots,d\right\}$, then 
    \begin{equation}\label{thm_coeff_repr_eq1}
      C(w)=0.  
    \end{equation}
    
    \noindent If $n^w(i)$ is even for every $i\in\{1,\dots,d\}$, then 
    \begin{multline}\label{thm_coeff_repr_eq2}
      C(w)=\displaystyle{\sum_{\vec{k}=(k_1,\dots, k_M)\in\mathcal{K}_\ell(M)}
	\frac{1}{k_1!\cdots k_M!}
	\prod_{j=1}^M\left(c_j\right)^{N^w\left(0,j,\vec{k}\right)}}\\
      \times
      \displaystyle{\prod_{p=1}^{d}
        \left(\sum_{\left\{d_{ij}\right\}_{1\leq i\leq j\leq M}\in 
	  e\left(N^w\left(p,1,\vec{k}\right),\dots,N^w\left(p,M,\vec{k}\right)
	  \right)}
	2^{-\sum_{i=1}^Md_{ii}}
	\frac{\prod_{j=1}^M\left(N^w\left(p,j,\vec{k}\right)!\right)}
	     {\prod_{1\leq i\leq j\leq M}\left(d_{ij}!\right)}
	     \prod_{1\leq i\leq j\leq M}R_{ij}^{d_{ij}}\right)}
    \end{multline}
    where $c_j$ and $R_{ij}$ are real numbers defined
    in~\eqref{Z_coeff_rvs}.
  \end{theorem}
  \begin{proof}
    In the case where $n^w(i)$ is odd for some
    $i\in\left\{1,\dots,d\right\}$,  
    \eqref{thm_coeff_repr_eq1} is directly derived 
    from~\eqref{Z_coeff_rvs}. 
    \par
    We therefore consider the other case.     
    By the Taylor expansion of
    $\exp\left(Z_1\right)\cdots\exp\left(Z_M\right)$, 
    we have
    \begin{equation*}
      E\left[\exp\left(Z_1\right)\cdots\exp\left(Z_M\right)\right]
      =\sum_{k_1,\dots, k_M=0}^\infty\frac{1}{k_1!\cdots k_M!}
      E\left[\left(c_1v_0+\sum_{i=1}^dS^i_1v_i\right)^{k_1}\cdots
	\left(c_Mv_0+\sum_{i=1}^dS^i_Mv_i\right)^{k_M}\right].  
    \end{equation*}
    Hence
    \begin{equation}\label{thm_coeff_repr_1}
      \begin{split}
	C(w)&=\langle 
	E\left[\exp\left(Z_1\right)\cdots\exp\left(Z_M\right)\right],
	w\rangle\\
	&=\sum_{\vec{k}=(k_1,\dots,k_M)\in\mathcal{K}_\ell(M)}
	\frac{1}{k_1!\cdots k_M!}
	E\left[\underbrace{S^{i_1}_1\cdots S^{i_{k_1}}_1}_{k_1}
	  \underbrace{S^{i_{k_1+1}}_2\cdots 
	    S^{i_{k_1+k_2}}_2}_{k_2}
	  \cdots
	  \underbrace{S^{i_{k_1+\cdots+k_{M-1}+1}}_M\cdots 
	    S^{i_{k_1+\cdots+k_M}}_M}_{k_M}\right]\\
	&=\sum_{\vec{k}=(k_1,\dots, k_M)\in\mathcal{K}_\ell(M)}
	\frac{1}{k_1!\cdots k_M!}
	E\left[
	  \left(c_1\right)^{N^w\left(0,1,\vec{k}\right)}
	  \cdots\left(c_M\right)^{N^w\left(0,M,\vec{k}\right)}
	  \left(S^1_1\right)^{N^w\left(1,1,\vec{k}\right)}\cdots
	  \left(S^1_M\right)^{N^w\left(1,M,\vec{k}\right)}\right.\\
	  &\qquad\qquad\qquad\qquad\qquad\qquad\qquad\left. \cdots
	  \left(S^d_1\right)^{N^w\left(d,1,\vec{k}\right)}\cdots
	  \left(S^d_M\right)^{N^w\left(d,M,\vec{k}\right)}\right].  
      \end{split}
    \end{equation}
    From the definition of $S^i_j$,  
    \begin{equation*}
      C(w)=
      \sum_{\vec{k}=(k_1,\dots,k_M)\in\mathcal{K}_\ell(M)}
      \frac{1}{k_1!\cdots k_M!}
      \prod_{j=1}^M\left(c_j\right)^{N^w\left(0,j,\vec{k}\right)}
      \prod_{p=1}^{d}
      E\left[\left(S^p_1\right)^{N^w\left(p,1,\vec{k}\right)}\cdots
	\left(S^p_M\right)^{N^w\left(p,M,\vec{k}\right)}\right]
    \end{equation*}
    Applying~\eqref{lem_coeff_repr_eq} from Lemma~\ref{lem_coeff_repr},  
    we obtain~\eqref{thm_coeff_repr_eq2}.
    \qed
  \end{proof}
  
  On the other hand, the value of the coefficient of each 
  $v_{i_1}\cdots v_{i_\ell}$ in\\
  $j_m\left(\exp\left(v_0+(1/2)\sum_{i=1}^dv_i^2\right)\right)$ can be
  obtained by the following proposition.
  
  \begin{myProp}
    Let $A^0=\{v_0, v_1v_1, v_2v_2, \dots, v_dv_d\}\subset A^*$.  
    Then 
    \begin{equation}\label{coeff_value}
      \exp\left(v_0+\frac{1}{2}\sum_{i=1}^dv_i^2\right)
      =\sum_{\substack{w=w_1\cdots w_l\\w_1,\dots,w_l\in A^0}}
      \frac{1}{2^{|w|-l}l!}w.
    \end{equation}
  \end{myProp}
  
  \noindent Therefore, taking 
  $\left\{S_j^i\right\}_{i=1,\dots, d, j=1,\dots, M}$ to
  equate~\eqref{thm_coeff_repr_eq1} or~\eqref{thm_coeff_repr_eq2} 
  with~\eqref{coeff_value} for
  $w=v_{i_1}v_{i_2}\cdots v_{i_\ell}$ with $\|w\|\leq m$, we can construct
  $Z_1,\dots, Z_M$.

  For $m=5$, we take $M=2$ to obtain solvable simultaneous
  equations which in fact become the following five:
  \begin{equation}\label{eqns_m_5_n_2}
    \begin{split}
      &c_1+c_2=1,\quad \frac{1}{2}(c_1R_{11}+c_2R_{22})+R_{12}=\frac{1}{2},\\
      &\frac{1}{6}(c_1R_{11}+c_2R_{22})+\frac{1}{2}c_1(R_{12}+R_{22})
      =\frac{1}{4},\\
      &\frac{1}{6}(c_1R_{11}+c_2R_{22})+\frac{1}{2}c_2(R_{11}+R_{22})
      =\frac{1}{4},\\
      &\frac{1}{24}(R_{11}^2+R_{22}^2)
      +\frac{1}{6}R_{12}(R_{11}+R_{22})+\frac{1}{4}R_{11}R_{22}
      =\frac{1}{8}.
    \end{split}
  \end{equation}
  The solution is~\eqref{soln_m_5_n_2}. Since we let 
  $\{S^i_j\}_{i=1,\dots, d, j=1,\dots, M}$ be 
  the Gaussian system, such random variables can be constructed.
  \begin{myRem}
    If we let $m=5$, then $M$ must be at least two.
  \end{myRem}

  \section{The Runge--Kutta method}\label{sec:3}
  
  \noindent We begin by briefly introducing the tree theory 
  following~\cite{bollobas:1979book}, ~\cite{Butcher:1987}, 
  and~\cite{Butcher:2003}. 
  For details of the Runge--Kutta method,
  see~\cite{Butcher:1987}, \cite{Butcher:2003}, and~\cite{roessler:2003}.  
  
  All trees introduced here are called directed or rooted trees 
  in the literature listed above.
  
  \begin{myDef}
    A labelled tree $\ltree{t}$ is a pair of finite sets
    $\left(V(\ltree{t}), E(\ltree{t})\right)$ that satisfies the
    following conditions:
    \begin{enumerate}[{\rm (1)}]
    \item $V(\ltree{t})\subset{\mathbb Z}$, $V(\ltree{t})\neq\emptyset$, 
      and $E(\ltree{t})\subset
      \left\{(x,y)\in V(\ltree{t})\times V(\ltree{t})\; : 
      \;x< y\right\}$.
    \item For each $x\in V(\ltree{t})$, 
      if $(x, y)\in E(\ltree{t})$
      and $(x', y)\in E(\ltree{t})$, then $x=x'$.\\
    \item For two distinct elements $x, y\in V(\ltree{t})$, 
      one of the followings holds:
      \begin{enumerate}[{\rm (i)}]
      \item There exists a path from $x$ to $y$.
      \item There exists a path from $y$ to $x$.
      \item For some $z\in V(\ltree{t})\setminus\{x, y\}$,
	there exist paths $z$ to $x$ and $z$ to $y$.
      \end{enumerate}
      Here a path from $p_1$ to $p_l$ is a sequence
      $(p_1, p_2), (p_2, p_3),\dots,(p_{l-1}, p_l)$
      of elements of $E(\ltree{t})$.
      
    \end{enumerate}
    An element of $V(\ltree{t})$ is called a vertex of $\ltree{t}$
    and that of $E(\ltree{t})$ is called an edge of $\ltree{t}$.
    \par
    A particular labelled tree $\tau_\ell$ is that
    with $\card\left(V(\tau_\ell)\right)=1$ 
    and $E(\tau_\ell)=\emptyset$.
    \par
    For a labelled tree $\ltree{t}=(V(\ltree{t}), E(\ltree{t}))$, 
    let $\ltree{r}(\ltree{t})$ be $\card\left(V(\ltree{t})\right)$. 
    We define $\ltree{T}$ as the set of all labelled trees.
  \end{myDef}
  \begin{myProp}
    For each $\ltree{t}=(V(\ltree{t}), E(\ltree{t}))$, 
    there exists a unique vertex $r\in V(\ltree{t})$ 
    such that for any $x\in V(\ltree{t})\setminus\{r\}$, there is a path 
    from $r$ to $x$.
  \end{myProp}
  Such a vertex $r$ is called the root of $\ltree{t}$.
  Here, $\tau_\ell$
  consists of only the root. 
  \begin{myDef}
    For $i=1,\dots, n$, let $\ltree{t}_i=\left(V\left(\ltree{t}_i\right), 
    E\left(\ltree{t}_i\right)\right)\in\ltree{T}$ be such that 
    $V\left(\ltree{t}_i\right)\cap V\left(\ltree{t}_j\right)=\emptyset$ 
    if $i\neq j$.
    Then $\mathbf{[}\ltree{t}_1\cdots\ltree{t}_n\mathbf{]}$ is 
    defined as
    $\ltree{t}=\left(V(\ltree{t}), E(\ltree{t})\right)\in\ltree{T}$ 
    such that 
    \begin{equation*}\begin{split}
	V(\ltree{t})&=\{r\}\cup V\left(\ltree{t}_1\right)\cup
	\cdots\cup V\left(\ltree{t}_n\right)\\
	E(\ltree{t})&=\{(r,r_1),\dots,(r, r_n)\}\cup E\left(\ltree{t}_1\right)
	\cup\cdots\cup E\left(\ltree{t}_n\right)
    \end{split}\end{equation*}
    where each $r_i$ 
    denotes the root of $\ltree{t}_i$
    and $r=\min\{r_1,\dots,r_n\}-1$.
  \end{myDef}
  \begin{myRem}
    For $\ltree{t}_1,\dots,\ltree{t}_n\in\ltree{T}$, we have that 
    \begin{equation*}
      \left[\ltree{t}_1\cdots\ltree{t}_n\right]=
      \left[\ltree{t}_{\varpi(1)}\cdots\ltree{t}_{\varpi(n)}\right]
    \end{equation*}
    for any permutation $\varpi\in\mathfrak{S}_n$. 
  \end{myRem}
  \begin{myDef}
    Let $\ltree{t}_i=\left(V\left(\ltree{t}_i\right), 
    E\left(\ltree{t}_i\right)\right)
    \in\ltree{T}$ for $i=1,2$.  
    We say that $\ltree{t}_1$ and $\ltree{t}_2$ are isomorphic, 
    written as $\ltree{t}_1\isom\ltree{t}_2$, 
    if there exists a bijection 
    $\varpi :\, V\left(
    \ltree{t}_1\right)\longrightarrow V\left(\ltree{t}_2\right)$
    such that $(x,y)\in E\left(\ltree{t}_1\right)$ 
    if and only if 
    $\left(\varpi(x), \varpi(y)\right)\in E\left(\ltree{t}_2\right)$.
    
    In particular, when $\ltree{t}_1\isom\ltree{t}_2$ and 
    $V\left(\ltree{t}_1\right)= V\left(\ltree{t}_2\right)$, 
    that is, $\varpi$ is a permutation, we say that $\ltree{t}_1$ and 
    $\ltree{t}_2$ are equivalent and write $\ltree{t}_1\sim\ltree{t}_2$.
  \end{myDef}
  \begin{myProp}
    Both $\isom$ and $\sim$ are equivalence relations.
  \end{myProp}
  \begin{myProp}\label{RK:prop:3}
    Let 
    $\ltree{t}_i=\left(V\left(\ltree{t}_i\right), 
    E\left(\ltree{t}_i\right)\right)\in\ltree{T}$ and 
    $\ltree{u}_i=\left(V\left(\ltree{u}_i\right), 
    E\left(\ltree{u}_i\right)\right)\in\ltree{T}$ for $i=1,\dots, n$.  
    Suppose that $\ltree{t}_i\isom\ltree{u}_i$ for $i=1,\dots,n$ and that 
    \begin{equation*}
      V\left(\ltree{t}_i\right)\cap V\left(\ltree{t}_j\right)=\emptyset 
      \quad\text{and}\quad 
      V\left(\ltree{u}_i\right)\cap V\left(\ltree{u}_j\right)=\emptyset 
    \end{equation*}
    if $i\neq j$.  Then 
    \begin{equation*}
      \mathbf{[}\ltree{t}_1\cdots\ltree{t}_n\mathbf{]}\isom
      \mathbf{[}\ltree{u}_1\cdots\ltree{u}_n\mathbf{]}.
    \end{equation*}
  \end{myProp}
  \begin{myDef}
    We define $T=\ltree{T}/\isom$.  
    An element $t\in T$ is called a non-labelled tree.  
    For a labelled tree $\ltree{t}\in\ltree{T}$, $|\ltree{t}|$ denotes the 
    corresponding non-labelled tree $t\in T$.
  \end{myDef}
  
  Then, from Proposition~\ref{RK:prop:3}, the following result can be 
  derived. 
  \begin{myProp}\label{prop:TreeConcat}
    Under the same condition as Proposition~\ref{RK:prop:3},
    \begin{equation*}
      \left|\left[\ltree{t}_1\cdots\ltree{t}_n\right]\right|
      =\left|\left[\ltree{u}_1\cdots\ltree{u}_n\right]\right|
    \end{equation*}
    holds.
  \end{myProp}
  By virtue of Proposition~\ref{prop:TreeConcat},
  we can define a non-labelled tree 
  $t=\left[t_1\cdots t_n\right]$ for $t_1,\dots, t_n\in T$ as
  $\left|\ltree{[}\ltree{t}_1\cdots\ltree{t}_n\ltree{]}\right|$ 
  where
  $\ltree{t}_i\in\ltree{T}$ is a representative labelled tree such that 
  $|\ltree{t}_i|=t_i$.  In particular, we let $\tau=|\tau_\ell|$.
  
  \begin{myProp}
    For any $t\in T\setminus\left\{\tau\right\}$, there exist 
    $t_1,\dots,t_n\in T$ such that 
    $t=\left[t_1\cdots t_n\right]$.  Moreover 
    \begin{equation*}
      \left[
	t_1\cdots t_n\right]=\left[t_{\varpi(1)}\cdots t_{\varpi(n)}\right]
    \end{equation*}
    for any permutation $\varpi\in\mathfrak{S}_n$. 
  \end{myProp}
  
  Here, $[t_1^{m_1}\cdots t_n^{m_n}]$ denotes 
  $[\underbrace{t_1\cdots t_1}_{m_1} \cdots \underbrace{t_n\cdots t_n}_{m_n}]$ 
  where $t_i\in T$ for $i=1,\dots,n$.
  
  \begin{myDef}
    \begin{enumerate}[{\rm(1)}]
    \item    For $t=(V(t), E(t))\in T$, we define 
      $\alpha:T\longrightarrow{\mathbb Z}_{\geq 1}$, 
      $r:T\longrightarrow{\mathbb Z}_{\geq 1}$, and 
      $\sigma:T\longrightarrow{\mathbb Z}_{\geq 1}$ by 
      \begin{equation*}\begin{split}
	  \alpha(t)&=\card\left(\left\{\ltree{u}\in\ltree{T} \,\middle\vert\;
	  \ltree{u}\sim\ltree{t}\;\text{where}\;\ltree{t}\in\ltree{T}\;
	  \text{is a representative element such that 
	    $|\ltree{t}|=t$}\right\}\right)\\
	  r(t)&=\card\left(V(t)\right)\\
	  \sigma(t)&=\begin{cases}
	  1 &\text{if $t=\tau$}\\
	  \prod_{i=1}^l m_i!\sigma\left(t_i\right)^{m_i} 
	  &\text{if $t=\left[t_1^{m_i}\cdots t_l^{m_l}\right],\ l\geq 1$}
	  \end{cases}
      \end{split}\end{equation*}
      where $A=\left(a_{ij}\right)_{i,j=1,\dots,K}$.  We notice that 
      $\alpha$ is well-defined because 
      $\alpha$ denotes the number of ways a tree may be labelled.
    \item
      Let $\mathcal{A}$ be the set of $K\times K$ real matrices.  
      We inductively define derivative weights 
      $\zeta_i: T\times\mathcal{A}\longrightarrow
      {\mathbb R}$ for $i=1,\dots,K$ 
      by 
      \begin{equation*}
	\zeta_i(t;A)
	=\begin{cases}
	\sum_{j=1}^{K}a_{ij}
	&\text{if $t=\tau$}\\ 
	\sum_{j=1}^Ka_{ij}\prod_{k=1}^l\zeta_j\left(t_k;A\right) 
	&\text{if $t=\left[t_1\cdots t_l\right]$, $l\geq 1$}. 
	\end{cases}      
      \end{equation*}
      In addition, we define the elementary differentials 
      $D:C_{b}^\infty({\mathbb R}^N;{\mathbb R}^N)\times T\longrightarrow 
      C_{b}^\infty({\mathbb R}^N;{\mathbb R}^N)$ 
      as follows: 
      \begin{equation}
	D(W,t)(x)=\begin{cases}
	W(x) &\text{if $t=\tau,$}\\
	W^{(l)}(x)\left(D(W,t_1)(x), D(W, t_2)(x), \dots,
	D(W,t_l)(x)\right)
	&\text{if $t=\left[t_1t_2\cdots t_l\right]$, $l\geq 1$.}
	\end{cases}
      \end{equation}
    \end{enumerate}
  \end{myDef}
  
  Let $y(W,s)$ be a solution to an ODE 
  \begin{equation}\label{ODE}
    \frac{d}{ds}y(W,s) = W\left(y(W,s)\right),
    \quad y(W,0) = y_0
  \end{equation}
  where $W\in C^\infty_b({\mathbb R}^N,
  {\mathbb R}^N)$ and $y_0\in{\mathbb R}^N$.  Then 
  we have the following lemmas essentially proved
  in~\cite{Butcher:2003}, pp.~139--145. 
  
  \begin{myLem}\label{RK_lem1}
    For $m\in{\mathbb Z}_{\geq 1}$, 
    \begin{equation}\label{RK_lem1_eq}
      y^{(m)}(W,s)=
      \sum_{\ltree{t}\in\ltree{T}_m}D\left(W,|\ltree{t}|\right)
      \left(y(W,s)\right).
    \end{equation}
  \end{myLem}
  
  Let 
  $T_m=\left\{t\in T : r(t)=m\right\}$ 
  and $T_{\leq m}=\bigsqcup_{n=0}^mT_m$ for $m\geq 0$ with $T_0=\emptyset$.  
  
  Let $A_{ex}\in\mathcal{A}$ denote $A$ in~\eqref{original_RK} 
  for the explicit Runge--Kutta method.  
  Then $Y_i(W,s)$ is definitely determined by $A_{ex}$ 
  with $a_{ij}=0$ if $i\leq j$ and so $Y(y_0;W,s)$ can be 
  constructed with $b$ and $Y_i(W,s)$
  as both seen in~\eqref{original_RK}.     
  \begin{myLem}\label{RK_lem2}
    Let $m\geq 1$.  If there exists a constant $C_m>0$ such that 
    \begin{equation}\label{RK_lem2_eq1}
      \left|Y_i(W,s)-
      \left(y_0+\sum_{t\in T_{\leq m-1}}s^{r(t)}
      \frac{\zeta_i(t)}{\sigma(t)}D(W,t)\left(y_0\right)\right)
      \right|
      \leq C_ms^m\|W\|^m_{C^m}
    \end{equation}
    for $i=1,\dots, K$, then there exists a constant $C_{m+1}$
    \begin{equation}\label{RK_lem2_eq2}
      \left|sW\left(Y_i(W,s)\right)
      -\sum_{\substack{l=1,\dots,m-1\\t=[t_1\cdots t_l]\in T_{\leq m}}}
      s^{r(t)}
      \frac{\prod_{k=1}^l\zeta_i(t_k)}
	   {\sigma(t)}D(W,t)\left(y_0\right)\right|
	   \leq C_{m+1}s^{m+1}\|W\|^{m+1}_{C^{m+1}}.
    \end{equation}
  \end{myLem}
  
  Applying these lemmas to evaluations of the solution 
  to~\eqref{ODE} and the Runge--Kutta method~\eqref{original_RK}, 
  we obtain the following result.
  \begin{theorem}\label{RK_Taylor_thm}
    For $y$ satisfying~\eqref{ODE}, there exists a constant $C_{m+1}$
    \begin{equation}\label{extended_Taylor_exp}
      \left|\exp{(sW)}(y_0)-\left(y_0+\sum_{t\in T_{\leq m}}
      \frac{s^{r(t)}}{r(t)!}\alpha(t)D(W,t)\left(y_0\right)\right)\right|
      \leq C_{m+1}s^{m+1}\|W\|_{C^{m+1}}^{m+1}.
    \end{equation}
    On the other hand, for the Runge--Kutta method~\eqref{original_RK} 
    there exists a constant $C^\prime_{m+1}$ such that 
    \begin{multline}\label{extended_RK_Taylor_exp}
      \left|Y(y_0;W,s)-\left(y_0+\sum_{\substack{l=1,\dots,m-1\\
	  t=[t_1\cdots t_l]\in T_{\leq m}}}
      \frac{s^{r(t)}}{\sigma(t)}\sum_{i=1}^Kb_i\prod_{k=1}^l
      \zeta_i\left(t_k;A\right)D(W,t)
      \left(y_0\right)\right)\right|\\
      \leq C^\prime_{m+1}s^{m+1}\|W\|_{C^{m+1}}^{m+1}.
    \end{multline}
  \end{theorem}
  
  We say that $(A,b)$ satisfies $m$-th-order conditions if  
  \begin{equation}\label{coeff_condis}
    \frac{\alpha(t)}{r(t)!}=
    \frac{\sum_{i=1}^Kb_i\prod_{k=1}^l\zeta_i\left(t_k;A\right)}{\sigma(t)} 
  \end{equation}
  for all $t=\left[t_1\cdots t_l\right]\in T_{\leq m}$.   
  
  From Theorem~\ref{RK_Taylor_thm}, the following result can be 
  directly derived. 
  \begin{theorem}\label{krk_main_thm3}
    Suppose that $(A,b)$ satisfies 
    the $m$th-order conditions~\eqref{coeff_condis}.  
    Let $g(W)(y_0)=Y(y_0;W,1)$ where $Y(y_0;W,1)$ is the Runge--Kutta 
    method defined in~\eqref{original_RK}.  Then 
    \begin{equation}\label{krk_main_thm3_eq}
      g\in\mathcal{IS}(m).
    \end{equation}
  \end{theorem}

  \section{The new simulation scheme and
    Corollary~\ref{contp_krk_main_cor}}\label{sec:4}
  \noindent Corollary~\ref{contp_krk_main_cor} indicates 
  the new implementation method of the new higher-order scheme 
  proposed by Kusuoka in~\cite{kusuoka:2001aprx} and~ \cite{kusuoka:2003}.  
  \par
  This implementation method seems to
  be distinct mainly because it has two advantages.  One is that the
  approximation operator can be obtained by numerical calculations if
  the Runge--Kutta method is applied to the calculation
  of each $\exp{(Z_j)}$ whereas the tediousness in
  symbolical calculations of the operator might be an obstacle for
  practical application, which can be observed
  in~\cite{KusuokaNinomiya:2004}, \cite{ninomiya:2003},
  and~\cite{shimizu:2002}.
  The other advantage is that the partial sampling
  problem discussed in~\cite{KusuokaNinomiya:2004}
  and~\cite{ninomiya:2003} can be resolved by using quasi-Monte Carlo
  methods.  More precisely, the following two points
  make an effective use of 
  the Low-Discrepancy sequences, which are essential to
  quasi-Monte Carlo methods(\cite{niederreiter:1992book}):
  \begin{itemize}
  \item In this implementation, $S_{j}^{i}$ can be taken to
    be a continuous random variable.
  \item The scheme itself is characterized by the need for a much less
    number discretization time steps, which leads to a reduction in the
    number of dimensions of the numerical integration.
  \end{itemize}

  \section{Application}\label{sec:5}
  \noindent In this section
  we present a numerical example
  in order to illustrate the implementation method 
  proposed in Corollary~\ref{contp_krk_main_cor} and
  compare it with some existing schemes. 
  \subsection{Simulation}
  \noindent Let $X(t,x)$ be a diffusion process defined by~\eqref{SDE}.
  \noindent The most popular scheme of first order is the Euler--Maruyama
  scheme, which is shown in~\cite{KloedenPlaten:1999}
  and~\cite{TalayTubaro:1990},
  for an arbitrary $C^4$ function $f$
  \begin{equation}\label{eq:EM}
    \left\Vert E\left[ f\left(\EMn{X}{n}_1\right) \right]
    -E\left[f\left(
      X(1,x)\right) \right] \right\Vert \leq C_{f}\frac{1}{n}
  \end{equation}
  where $\EMn{X}{n}_1$ denotes the Euler--Maruyama scheme approximating
  $X(t,x)$.  We note that this inequality holds for measurable $f$
  if $\left\{V_i\right\}_{i=1,\dots,d}$ satisfies some more conditions
  (\cite{BallyTalay:1996}). 
  \par
  The construction of a higher-order scheme is based on the higher order
  stochastic Taylor
  formula~(\cite{Castell:1993}\cite{KloedenPlaten:1999}).  When the
  vector fields $\left\{V_i\right\}_{i=0}^d$ commute, higher-order
  schemes can be simplified to a direct product 
  of one-dimensional problem as seen in~\cite{KloedenPlaten:1999}.  
  In contrast, for non-commutative $\left\{V_i\right\}_{i=0}^d$,
  the acquisition of all iterated integrals of
  Brownian motion is required, which is very demanding.  This is done
  in~\cite{kusuoka:2001aprx},\cite{LiuLi:2000},\cite{Talay:1990},
  \cite{Talay:1995} and~\cite{KusuokaNinomiya:2004} and generalized as
  the cubature method on Wiener space~(\cite{LyonsVictoir:2002}).
  
  Once a $p$th-order scheme $\{\ODp{X}{n}_{k/n}\}_{k=0,\dots,n}$
  is obtained and  expanded with some constant $K_f$ as
  \begin{equation}\label{eq:RE1}
    E\left[f\left(\ODp{X}{n}_{1}\right)\right]
    -E\left[f\left(X(1,x)\right)\right]
    = K_f\frac{1}{n^p} + O\left(\frac{1}{n^{p+1}}\right),
  \end{equation}
  the $(p+1)$th-order scheme can be derived as
  \begin{equation}\label{eq:RE2}
    \frac{2^{p}}{2^{p}-1}E\left[f\left(\ODp{X}{2n}_{1}\right)\right]-\frac{1
    }{2^{p}-1}E\left[f\left(\ODp{X}{n}_{1}\right)\right].
  \end{equation}
  This boosting method is called Romberg extrapolation and is shown to
  be applicable to the Euler--Maruyama scheme under certain
  conditions~(\cite{TalayTubaro:1990}).
  \par
  The simulation approach must be followed by the numerical
  calculation of
  $\displaystyle{
    E\left[
      f\left(
      \ODp{X}{n}_1
      \right)
      \right]
  }$.
  However, when $n\times d$ is large,
  it is practically impossible to proceed with the integration
  by using the trapezoidal formula
  and so we fall back on the Monte Carlo or quasi-Monte Carlo
  method~(\cite{niederreiter:1992book}).
  Here we make only a few remarks on
  each method.
  For a more detailed analysis, see~\cite{NinomiyaVictoir:2005}.
  \begin{myRem}\label{rem:MC}
    As long as we use the Monte Carlo method for numerical approximation
    of $E[f(X(1,x))]$, the number of sample points needed to attain a
    given accuracy is independent of the number of the dimensions of
    integration, namely both the number $n$ of partitions and the order $p$
    of the approximation scheme.
  \end{myRem}
  
  \begin{myRem}\label{rem:QMC}
    In contrast to the Monte Carlo case, the number of sample points
    needed for the quasi-Monte Carlo method for numerical approximation
    of $E[f(X(1,x))]$ heavily depends on the number of the dimensions of
    integration.
    The fewer the dimensions, the fewer the samples that are needed.
  \end{myRem}
  
  \subsection{The algorithm and competitors}\label{ss:TheAlgorithm}
  \subsubsection{The algorithm of the new method}
  \noindent
  We take the algorithm which is proposed in
  Theorem~\ref{contp_krk_main_thm2} and
  Corollary~\ref{contp_krk_main_cor} with $u=3/4$.
  From 
  Corollary~\ref{contp_krk_main_cor}, we can implement the second-order
  algorithm with a numerical approximation of $\exp\left(Z_i\right)$ of
  at least fifth-order Runge--Kutta method 
  because the order $m$ for an integration scheme attained 
  by $Z_1$ and $Z_2$ is five and so the order of the new 
  implementation method becomes two.
  As a result of the same argument it can be shown that 
  at least seventh-order explicit Runge--Kutta method has to be 
  applied to the approximation of $\exp\left(Z_i\right)$ 
  when we boost the new method to the third order
  by Romberg extrapolation.
  Details of these Runge--Kutta algorithms used here are
  given in the Appendix.
  
  \subsubsection{Competitive schemes}
  \noindent 
  There there are numerous studies on the acceleration of Monte Carlo
  methods~(\cite{Glasserman:2004}). We choose for the following reasons
  only the crude Euler--Maruyama scheme and the algorithm
  introduced in~\cite{NinomiyaVictoir:2005}, which we will refer
  to in the remainder of this paper as N-V method,
  both with
  and without Romberg extrapolation, as competitors:
  \begin{enumerate}[(i)]
  \item Only these two schemes can be recognized as being comparable to the
    new method, since they are model-independent.
  \item Almost all variance reduction techniques
    and dimension reduction techniques applicable to the
    Euler--Maruyama scheme are also applicable to the new method.
  \end{enumerate}
  \noindent
  
  \subsection{Numerical results}
  \noindent We provide an example on financial option
  pricing in the following part of this paper.
  \subsubsection{Asian option under the Heston model}  
  \noindent We consider an Asian call option written on an asset whose
  price process follows the Heston stochastic volatility model.
  Comparison with the N-V
  method will also be given as well from the result shown
  in~\cite{NinomiyaVictoir:2005}.  \par The non-commutativity of this
  example should be noted here.
  
  \noindent Let $Y_1$ be the price process of an asset following the
  Heston model:
  \begin{equation}\label{eq:heston}
    \begin{split}
      Y_1(t, x) =& x_1+\int_0^t\mu Y_1(s, x)\,ds
      +\int_0^t Y_1(s, x)\sqrt{Y_2(s,x)}\,dB^1(s), \\
      Y_2(t, x) =& x_2+\int_0^t\alpha
      \left(\theta - Y_2(s, x)\right)\,ds\\
      &\quad+\int_0^t\beta
      \sqrt{Y_2(s, x)}\left(\rho\,dB^1(s)+\sqrt{1-\rho^2}\,dB^2(s)\right),
    \end{split}
  \end{equation}
  where $x = (x_1, x_2) \in(\mathbb{R}_{>0})^2$,
  $(B^1(t), B^2(t))$ is a two-dimensional standard Brownian motion,
  $-1\leq\rho\leq 1$,
  and $\alpha$, $\theta$, $\mu$ are some positive coefficients such that
  $2\alpha\theta-\beta^2>0$ to ensure the existence and uniqueness of a
  solution to the stochastic differential equation (\cite{feller:1950}).
  Then the payoff of Asian call option on this asset with maturity
  $T$ and strike $K$ is $\max\left(Y_3(T,x)/T-K, 0\right)$
  where
  \begin{equation}\label{eq:heston2}
    Y_3(t,x) = \int_0^tY_1(s,x)\,ds.
  \end{equation}
  Hence, the price of this option becomes $D\times
  E\left[\max\left(Y_3(T,x)/T - K,\, 0\right)\right]$ where $D$ is an
  appropriate discount factor that we do not focus on here.
  We set $T=1$, $K=1.05$, $\mu = 0.05$, $\alpha = 2.0$, $\beta = 0.1$,
  $\theta = 0.09$, $\rho=0$, and $(x_1, x_2) = (1.0, 0.09)$ and take
  $$E\left[\max\left(Y_3(T,x)/T - K,\, 0\right)\right] = 6.0473534496
  \times 10^{-2}$$ that is obtained by the new method with Romberg
  extrapolation and the quasi-Monte Carlo with $n=96+48$, and $M=8\times
  10^8$ where $M$ denotes the number of sample points.  \par Let $Y(t,x) =
  \big{.}^t\!\left(Y_1(t,x), Y_2(t,x), Y_3(t,x)\right)$.  Transformation
  of the stochastic differential equations~\eqref{eq:heston}
  and~\eqref{eq:heston2} gives the
  following Stratonovich-form stochastic differential equations:
  \begin{equation}\label{eq:heston3}
    Y(t,x) = \sum_{i=0}^2\int_0^tV_i(Y(s,x))\circ dB^i(s),
  \end{equation}
  where
  \begin{equation}\begin{split}
      V_0\left({}^t\!\left(y_1, y_2, y_3\right)\right) &=
      \bigg{.}^t\!
      \left(y_1\left(\mu-\frac{y_2}{2}-\frac{\rho\beta}{4}\right),\,
      \alpha(\theta - y_2)-\frac{\beta^2}{4},\,
      y_1\right) \\
      V_1\left({}^t\!\left(y_1, y_2, y_3\right)\right) &=
      \Big{.}^t\!
      \left(y_1\sqrt{y_2},\, \rho\beta\sqrt{y_2},\, 0\right) \\
      V_2\left({}^t\!\left(y_1, y_2, y_3\right)\right) &=
      \Big{.}^t\!
      \left(0,\, \beta\sqrt{\left(1-\rho^2\right)y_2},\, 0\right).
  \end{split}\end{equation}
  
  \subsubsection{Dimensions of integrations}
  \noindent As mentioned in Remarks~\ref{rem:MC} and~\ref{rem:QMC}, the
  dimensions of integrations in these methods affect
  the quasi-Monte Carlo method.  The relation among $d$: the number of
  factors, $n$: the number of partitions, and the dimensions of
  integration of each method can be summarized as in Table~\ref{table:dim}.
  \begin{table}
    \caption{\# of dimensions involved in each method.}
    \label{table:dim}       
    \begin{tabular}{lc}
      \hline\noalign{\smallskip}
      Method & Number of dimensions \\
      \noalign{\smallskip}\hline\noalign{\smallskip}
      \noalign{\smallskip}\hline\noalign{\smallskip}
      Euler--Maruyama & $dn$ \\
      \noalign{\smallskip}\hline\noalign{\smallskip}
      N-V & $n+dn$ ($n$-Bernoulli and $(d\times n)$-Gaussian)\\
      \noalign{\smallskip}\hline\noalign{\smallskip}
      New Method & $2dn$ \\
      \noalign{\smallskip}\hline
    \end{tabular}
  \end{table}
  
  \subsubsection{Discretization Error}
  \noindent The relation between discretization error
  and the number of partitions of each
  algorithm is plotted in Figure~\ref{fig:disc-err}.  We can observe
  from this figure that for $10^{-4}$ accuracy the new method with
  Romberg extrapolation takes the minimum number of partitions as $n=1+2$
  whereas $n=16$ for the Euler--Maruyama scheme with the extrapolation.  Even
  without the extrapolation, the new method attains that accuracy with
  $n=10$ while the Euler--Maruyama scheme takes $n=2000$.  Moreover,
  it may be said that the N-V method shows slightly worse performance
  than the new method.
  
  \begin{figure}
    \includegraphics{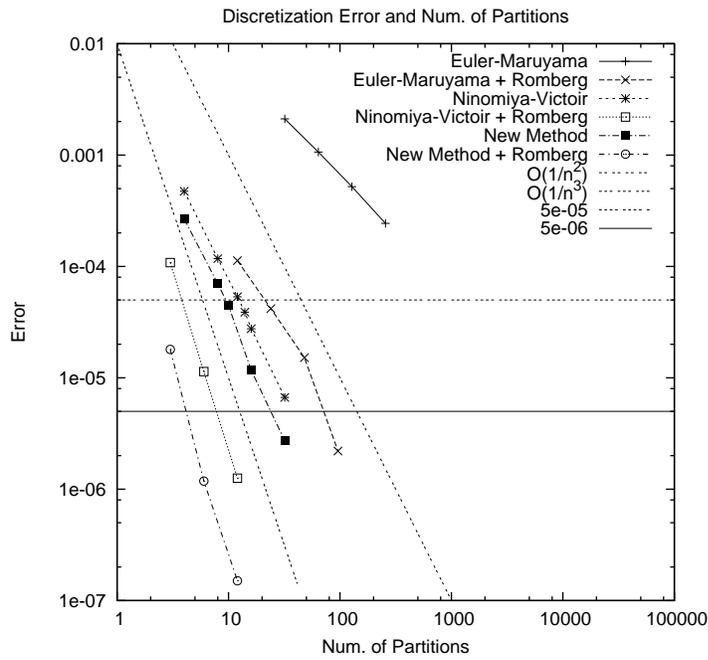}
    \caption{Error coming from the discretization}
    \label{fig:disc-err}       
  \end{figure}
  
  \subsubsection{Integration Error}
  
  \begin{figure}
    \includegraphics{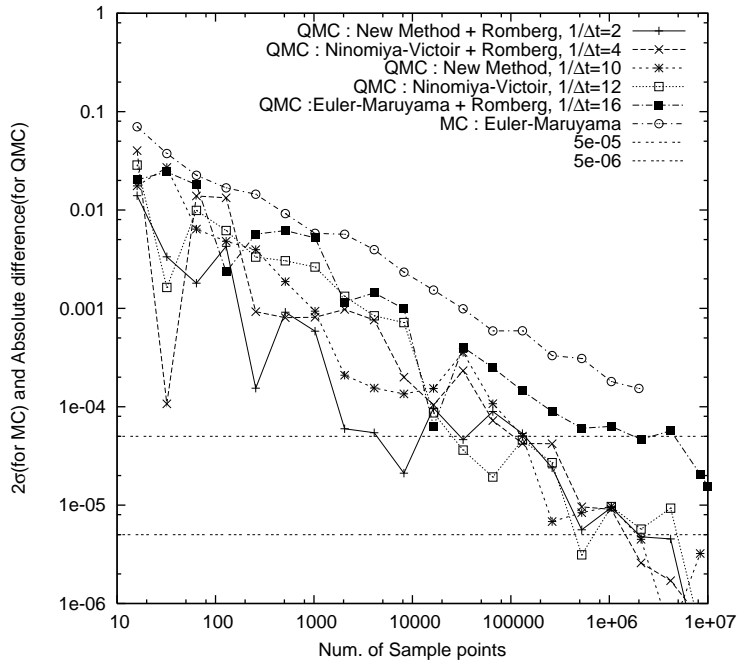}
    \caption{Convergence Error from quasi-Monte Carlo and Monte Carlo}
    \label{fig:QMC_conv}
  \end{figure}
  
  \noindent Looking at Figure~\ref{fig:QMC_conv}, we can compare
  convergence errors of respective methods for each number of sample
  points, $M$.  For the Monte Carlo case, $2\sigma$ of $10$ batches is taken
  as convergence error while for the quasi-Monte Carlo method, absolute
  difference from the value to be convergent is considered.  For
  $10^{-4}$ accuracy with $95\%$ confidence level ($2\sigma$), $M=10^8$
  is taken for the Monte Carlo method.  On the other hand, if we apply
  instead the quasi-Monte Carlo method, the new method and the N-V
  method require $M=2\times 10^5$ sample points, though $M=5\times 10^6$
  has to be taken for the Euler--Maruyama scheme.
  
  \subsubsection{Overall performance comparison}
  \begin{table}
    \caption{\#Partitions, \#Samples, Dimension, and CPU time required 
      for an accuracy of $10^{-4}$.}\label{table:table}
    \begin{tabular}{lcccc}
      \hline\noalign{\smallskip}
      Method & \#Part. & Dim. &\#Samples & CPU time (sec) \\
      \hline\noalign{\smallskip}
      \hline\noalign{\smallskip}
      E-M + MC & $2000$ & $4000$ & $10^8$ & $1.72\times 10^5$ \\
      \hline\noalign{\smallskip}
      E-M + Romb. + QMC & $16 + 8$ & $48$ &$5\times 10^6$
      & $1.27\times 10^2$ \\
      \hline\noalign{\smallskip}
      N-V + QMC & $16$ & $32+16$ &$2\times 10^5$ & $4.38$ \\
      \hline\noalign{\smallskip}
      N-V + Romb. + QMC & $4+2$ & $12+6$ & $2\times 10^5$ & $1.76$ \\
      \hline\noalign{\smallskip}
      New Method + QMC & $10$ & $40$ & $2\times 10^5$ & $3.4$ \\
      \hline\noalign{\smallskip}
      New Method + Romb. + QMC & $2+1$ & $12$ & $2\times 10^5$ & $1.2$ \\
      \hline\noalign{\smallskip}
    \end{tabular}
  \end{table}
  
  \noindent
  The number of partitions, the number of samples, and the amount of
  computation time required for $10^{-4}$ accuracy for each method are
  summarized in Table~\ref{table:table}.  CPU used in this experiment is
  Athlon 64 3800+ by AMD.
  
  \par Since the amount of time required to carry out the calculation
  for each sample point is proportional to the number of partitions,
  the total time spent on
  calculations is proportional both to the number of
  partitions and to the number of samples.
  We can see from the Table~\ref{table:table}
  that the speed of the new method is
  approximately $100$ times faster than that of the Euler--Maruyama
  scheme when Romberg extrapolation and quasi-Monte Carlo are applied to
  each.  Even when the extrapolation is not applied,
  the new method enables calculations some $37$ times faster
  than the Euler--Maruyama
  scheme with Romberg extrapolation and quasi-Monte Carlo method.
  This fact shows that the reduction in the number of partitions 
  sufficiently compensate for
  the slowness of one step of the new method at least in 
  the present study.
  
  \par Lastly, Remarks~\ref{rem:MC} and~\ref{rem:QMC} should be
  emphasized to reiterate that the advantage of the new method
  is that it is deeply
  related to the properties of the quasi-Monte Carlo method.
  
  \section*{Appendix: The fifth-order and the seventh-order
    Runge--Kutta algorithms}
  \noindent
  We present here the concrete algorithms of the explicit fifth-
  and seventh-order Runge--Kutta methods
  applied in Subsection~\ref{ss:TheAlgorithm}.
  The fifth-order method is taken from~\cite{Butcher:1987} as follows:
  \begin{gather*}
    a_{21} = \frac{2}{5},\quad a_{31} = \frac{11}{64},
    \quad a_{32}=\frac{5}{64},\quad a_{43}=\frac{1}{2},
    \quad a_{51} = \frac{3}{64},\quad a_{52} = -\frac{15}{64},\\
    a_{53} = \frac{3}{8}, \quad a_{54}= \frac{9}{16},
    \quad a_{62} = \frac{5}{7},\quad a_{63} = \frac{6}{7},
    \quad a_{64} = -\frac{12}{7},\quad a_{65} = \frac{8}{7}, \\
    a_{ij} = 0\quad\text{otherwise},\\
    b = \begin{pmatrix}
      \displaystyle{\frac{7}{90}} & 0
      & \displaystyle{\frac{32}{90}} & \displaystyle{\frac{12}{90}}
      & \displaystyle{\frac{32}{90}}
      & \displaystyle{\frac{7}{90}}
    \end{pmatrix}.
  \end{gather*}
  The seventh-order method is taken from~\cite{Butcher:2003} as follows:
  \begin{gather*}
    a_{21} = \frac{1}{6},\quad a_{32} = \frac{1}{3},
    \quad a_{41} = \frac{1}{8},\quad a_{43}=\frac{3}{8},
    \quad a_{51}=\frac{148}{1331},\quad a_{53}=\frac{150}{1331},
    \quad a_{54}= -\frac{56}{1331},\\
    a_{61}= -\frac{404}{243},
    \quad a_{63}= -\frac{170}{27},\quad a_{64}= \frac{4024}{1701},
    \quad a_{65}= \frac{10648}{1701},\quad a_{71}= \frac{2466}{2401},
    \quad a_{73}=\frac{1242}{343}, \\
    a_{74}= -\frac{19176}{16807},
    \quad a_{75}= -\frac{51909}{16807},\quad a_{76}=\frac{1053}{2401},
    \quad a_{81}=\frac{5}{154},\quad a_{84}=\frac{96}{539},
    \quad a_{85}= -\frac{1815}{20384},\\
    a_{86}= -\frac{405}{2464},\quad a_{87}=\frac{49}{1144},
    \quad a_{91}= -\frac{113}{32},
    \quad a_{93}= -\frac{195}{22},\quad a_{94}=\frac{32}{7},
    \quad a_{95}= \frac{29403}{3584},\\
    a_{96}= -\frac{729}{512},\quad a_{97}= \frac{1029}{1408},
    \quad a_{98} = \frac{21}{16},
    \quad a_{ij}=0 \quad\text{otherwise},\\
    b=\begin{pmatrix}
    0 & 0 & 0 & \displaystyle\frac{32}{105} &
    \displaystyle\frac{1771561}{6289920} &
    \displaystyle\frac{243}{1560} & 
    \displaystyle\frac{16807}{74880} &
    \displaystyle\frac{77}{1440}& \displaystyle\frac{11}{70}
    \end{pmatrix}.
  \end{gather*}

\bibliographystyle{spmpsci} \bibliography{mariko}   


\end{document}